\documentclass[11pt]{article}
\usepackage[T1]{fontenc}
\usepackage{amsfonts}

\usepackage{geometry}
\usepackage{bm}
\usepackage{cite}
\usepackage{latexsym}
\usepackage{enumerate,verbatim}

\usepackage[english]{babel}
\usepackage[utf8]{inputenc}
\usepackage{multirow}
\usepackage{amsmath,amsthm,amsfonts,amssymb}
\usepackage{mathtools}
\usepackage{graphics}
\usepackage{mathrsfs}

\usepackage{enumitem}

\usepackage[dvipsnames]{xcolor}
\usepackage{setspace}
\usepackage{needspace}
\definecolor{RevisionPurple}{RGB}{112,48,160}

\newtheorem{example}{Example}[section]
\usepackage{url}

\usepackage{hyperref}
\hypersetup{colorlinks=true, citecolor=MidnightBlue, linkcolor=MidnightBlue, urlcolor=MidnightBlue}

\usepackage{graphicx}
\usepackage{colortbl}
\usepackage{caption}

\usepackage{pifont}

\usepackage[ruled,linesnumbered]{algorithm2e}
\usepackage{appendix}

\newtheorem{theorem}{Theorem}[section]

\newtheorem{assumption}{Assumption}[section]
\newtheorem{lemma}{Lemma}[section]
\newtheorem{proposition}{Proposition}[section]

\theoremstyle{definition}
\newtheorem{definition}{Definition}[section]

\theoremstyle{remark}
\newtheorem{remark}{Remark}[section]

\theoremstyle{plain}
\newtheorem{fact}{Fact}[section]

\numberwithin{equation}{section}

\newcommand{\y}{\bm{y}}
\newcommand{\z}{\bm{z}}
\newcommand{\x}{\bm{x}}

\newcommand{\bz}{\mathbf 0}

\newcommand{\R}{\mathbb R}

\newcommand{\X}{\mathcal X}

\DeclareMathOperator{\proj}{proj}

\DeclareMathOperator{\dist}{dist}

\DeclareMathOperator*{\argmin}{argmin}
\DeclareMathOperator*{\argmax}{argmax}

\def\cA{{\cal A}}

\def\cC{{\cal C}}

\def\cI{{\cal I}}

\def\cL{{\cal L}}

\def\cN{{\cal N}}

\def\cX{{\cal X}}
\def\cY{{\cal Y}}

\def\a{\bm{a}}
\let\OriginalTextb\b
\DeclareRobustCommand{\b}{\ifmmode\bm{b}\else\expandafter\OriginalTextb\fi}
\let\OriginalTextc\c
\DeclareRobustCommand{\c}{\ifmmode\bm{c}\else\expandafter\OriginalTextc\fi}
\let\OriginalTextd\d
\DeclareRobustCommand{\d}{\ifmmode\bm{d}\else\expandafter\OriginalTextd\fi}
\def\e{\bm{e}}

\def\g{\bm{g}}

\let\OriginalTexti\i
\DeclareRobustCommand{\i}{\ifmmode\bm{i}\else\expandafter\OriginalTexti\fi}
\let\OriginalTextj\j
\DeclareRobustCommand{\j}{\ifmmode\bm{j}\else\expandafter\OriginalTextj\fi}
\let\OriginalTextk\k
\DeclareRobustCommand{\k}{\ifmmode\bm{k}\else\expandafter\OriginalTextk\fi}
\let\OriginalTextl\l
\DeclareRobustCommand{\l}{\ifmmode\bm{l}\else\expandafter\OriginalTextl\fi}

\let\OriginalTexto\o
\DeclareRobustCommand{\o}{\ifmmode\bm{o}\else\expandafter\OriginalTexto\fi}

\def\q{\bm{q}}
\let\OriginalTextr\r
\DeclareRobustCommand{\r}{\ifmmode\bm{r}\else\expandafter\OriginalTextr\fi}
\def\s{\bm{s}}
\let\OriginalTextt\t
\DeclareRobustCommand{\t}{\ifmmode\bm{t}\else\expandafter\OriginalTextt\fi}
\let\OriginalTextu\u
\DeclareRobustCommand{\u}{\ifmmode\bm{u}\else\expandafter\OriginalTextu\fi}
\let\OriginalTextv\v
\DeclareRobustCommand{\v}{\ifmmode\bm{v}\else\expandafter\OriginalTextv\fi}
\def\w{\bm{w}}
\def\x{\bm{x}}
\def\y{\bm{y}}
\def\z{\bm{z}}
\def\bz{{\mathbf 0}}

\def\1{{\mathbf 1}}

\newif\ifshowchanges
\showchangestrue

\title{{Nonconvex Composite Functional Constraints via First-Order Augmented Lagrangian Methods under Local Regularity}}
\author{Linglingzhi Zhu\thanks{H. Milton Stewart School of Industrial and Systems Engineering, Georgia Institute of Technology, Atlanta, GA, USA. (\href{mailto:llzzhu@gatech.edu}{llzzhu@gatech.edu})} \and Jiajin Li\thanks{Sauder School of Business, University of British Columbia, Vancouver, BC, Canada. (\href{mailto:jiajin.li@sauder.ubc.ca}{jiajin.li@sauder.ubc.ca})}}
\date{October 2026}

\allowdisplaybreaks[4]
\hypersetup{
  pdftitle={Nonconvex Composite Functional Constraints via First-Order Augmented Lagrangian Methods under Local Regularity},
  pdfauthor={Linglingzhi Zhu and Jiajin Li}
}

\begin{document}
\maketitle
\vspace{-3pt}
\begin{abstract}
We study nonasymptotic convergence of primal-dual methods for a class of nonconvex constrained optimization problems under local multiplier
regularity.  Both the objective and functional inequality constraints are compositions
of convex Lipschitz outer functions with smooth nonlinear mappings.
Unlike existing analyses requiring regularity globally or along the entire iterate sequence, we impose multiplier regularity only near the feasible set. Using an auxiliary minimax problem with a compact dual set, we derive Lyapunov estimates that identify nearly feasible iterates under feasible initialization and a sufficiently large penalty, without assuming multiplier boundedness. Local regularity then bounds the associated multipliers and makes dual truncation inactive, yielding approximate KKT certificates for the original problem. The smoothed prox-linear augmented Lagrangian method with
dual regularization yields an
$\mathcal O(K^{-1/3})$ KKT residual rate.
Without regularization, piecewise linear outer functions and
additional local structural assumptions yield a local dual error
bound and the sharper $\mathcal O(K^{-1/2})$ rate.
\end{abstract}



\section{Introduction}

We consider nonsmooth nonconvex constrained optimization problems whose
objective and functional inequality constraints have the convex-composite form
\begin{equation}\label{eq:problem}\tag{P}
\begin{array}{c@{\quad}l}
\min\limits_{\x\in\cX} & h_0(c_0(\x)) \\[0.2em]
{\rm s.t.} & h_i(c_i(\x))\leq 0,\quad i=1,\ldots,d.
\end{array}
\end{equation}
Here, \(\cX\subseteq\R^n\) is a nonempty compact convex set, each outer
function \(h_i\) is convex and Lipschitz continuous, and each inner mapping
\(c_i\) is smooth and possibly nonlinear. 
This convex-composite structure is classical in nonsmooth optimization and model-based first-order methods~\cite{lewis2016proximal,drusvyatskiy2019efficiency,davis2019stochastic}. It includes empirical conditional value-at-risk constraints based on smooth scenario losses~\cite{rockafellar2000optimization}, as well as maximum-type and finite robust constraints expressed as pointwise maxima of smooth scenario functions~\cite{ben2009robust,bertsimas2011theory}.

Projection onto the feasible set of \eqref{eq:problem} generally requires solving a nonsmooth nonconvex problem. This motivates first-order methods that avoid projection onto the full feasible set. Methods developed for nonsmooth weakly convex functional constraints include 
subgradient methods, inexact proximal point methods, and exact penalty schemes~\cite{ma2020quadratically,boob2023stochastic,huang2023oracle,jia2025first,liu2025single,yang2025single}. Quadratic penalty methods have also been studied for smooth nonlinear constraints, with corresponding complexity guarantees established in~\cite{lin2022complexity,alacaoglu2024complexity}. In this paper, we pursue a complementary multiplier-based approach centered on augmented Lagrangian methods (ALMs).

By incorporating the constraint functions into the objective through multiplier and penalty terms, ALMs avoid direct projection onto the generally nonconvex feasible set. For problems with smooth functional constraints, convergence and complexity guarantees have been established for various ALM-type methods~\cite{bolte2018nonconvex,xie2021complexity,he2023newton,kong2023iteration,li2024stochastic,zhu2024first,el2025complexity}. A central analytical issue is controlling constraint violation and multipliers, since the primal iterates need not remain feasible. For nonconvex functional constraints, existing complexity analyses often rely on global error bounds or Polyak--\L{}ojasiewicz-type conditions on constraint violation, uniform constraint qualifications, or explicit bounded-multiplier assumptions.  For
convex inequality constraints, smoothed
proximal Lagrangian, inexact Moreau envelope Lagrangian, and damped proximal
ALM methods obtain complexity guarantees by exploiting the convexity of the
constraint system together with Slater-type or local error-bound conditions
\cite{pu2024smoothed,huang2025inexact,dahal2026damped}.

The asymptotic ALM literature provides a more local perspective on multiplier
control. Conditions such as relaxed constant positive linear dependence and
relaxed quasinormality have been used to study asymptotic convergence and
multiplier boundedness
\cite{andreani2012relaxed,andreani2025relaxed}.
These results motivate the use of local regularity at feasible accumulation
points to control multipliers. A nonasymptotic analysis must additionally
identify nearly feasible iterates within a prescribed number of iterations.
For orthogonality constraints, the analysis in~\cite{zhu2026primal}
identifies such iterates through Lyapunov estimates and then uses the
geometry of the Stiefel manifold to bound the associated multipliers.
This geometric argument does not apply directly to the general functional
inequalities in~\eqref{eq:problem}.

This discussion leads to the following question:
\begin{center}
\textbf{Can local regularity provide the multiplier control needed for a
nonasymptotic primal-dual ALM analysis of nonconvex functional constraints?}
\end{center}
We answer this question affirmatively. The multiplier bounds supplied
by local regularity apply only near the feasible set. The key difficulty
is therefore to identify nearly feasible iterates within a prescribed
number of iterations without relying on those bounds.

To carry out this analysis, we introduce an auxiliary compact dual set
$\cY:=\{\y\in\R^d_+:\|\y\|_1\le R_{\y}\}$
and study the resulting minimax problem. This compactification allows
us to derive global Lyapunov estimates within the nonsmooth
nonconvex-concave framework of smoothed prox-linear descent--ascent
methods~\cite{li2025nonsmooth}, motivating the smoothed prox-linear
ALM studied here.  Our main technical contribution is to prove that the artificial upper
bound $\|\y\|_1\le R_{\y}$ is inactive at selected iterates. The proof
combines two counting arguments based on a Lyapunov sufficient decrease
estimate, one for constraint violation and the other for algorithmic
residuals. With feasible initialization and a sufficiently large penalty,
these arguments identify iterates that are nearly feasible for the
original problem and nearly stationary for the auxiliary minimax
problem. Local conic regularity then yields a uniform bound on the
corresponding multipliers. Choosing $R_{\y}$ larger than this bound
makes the artificial upper bound inactive, allowing us to convert
the stationarity estimates for the auxiliary problem into approximate
KKT certificates for~\eqref{eq:problem}.

The convergence rate is governed by the sufficient decrease estimate
for the Lyapunov function. Establishing this estimate requires
controlling a dual sensitivity term. We handle this term through
dual error bounds and consider two regimes.  We first introduce dual regularization with parameter $r_{\y}>0$,
following the smoothing perturbation viewpoint
of~\cite{li2026smoothing}. This makes the dual subproblem strongly
concave and yields a global dual error bound. We then balance the
regularization bias with the descent estimate to obtain an
$\mathcal O(K^{-1/3})$ bound on the KKT residual of the original
constrained problem after $K$ iterations. We next consider $r_{\y}=0$. This removes the regularization bias,
but the dual subproblem need not be strongly concave. Under additional
local structural assumptions, including piecewise linearity of the
outer functions, strict complementarity, and a piecewise analogue of
linear independence constraint qualification (LICQ), we establish
a local dual error bound. Our proof uses a piecewise linear epigraphical
sensitivity argument for the nonsmooth composite dual solution map.
This analysis is related to the work of Burke and
Engle~\cite{burke2020strong}, who use generalized equations to establish
strong metric subregularity of KKT mappings for piecewise
linear-quadratic convex-composite optimization.  We use our
local dual error bound to control the dual sensitivity term and
obtain the sharper $\mathcal O(K^{-1/2})$ rate.

\paragraph{Notation}
Throughout the paper, we use standard notation. For a
positive integer \(d\), let \([d]:=\{1,\ldots,d\}\). For a vector \(\a\in\R^d\), we write
$
(\a_+)_i:=\max\{a_i,0\}$, $i\in[d]$. 
For an index set \(\mathcal I\subseteq[d]\), \(\a_{\mathcal I}\) denotes
the subvector of \(\a\) indexed by \(\mathcal I\). Unless otherwise stated,
\(\|\cdot\|\) denotes the Euclidean norm for vectors and the induced
spectral norm for matrices, \(\|\cdot\|_1\) denotes the vector
\(\ell_1\)-norm. 
For a set \(\mathcal{C}\), we denote by \(\operatorname{int}(\mathcal{C})\) its interior, and
by
$
\dist(\x,\mathcal{C}):=\inf_{\u\in \mathcal{C}}\|\x-\u\|$
the distance from \(\x\) to \(\mathcal{C}\). If \(\mathcal{C}\) is nonempty, closed, and
convex, then
$
\proj_{\mathcal C}(\x):=\argmin_{\u\in \mathcal{C}}\|\u-\x\|$
denotes the projection onto \(\mathcal{C}\). We use
$
\mathbb B_{\varepsilon}(\x):=\{\u:\|\u-\x\|\le \varepsilon\}
$
for the closed Euclidean ball centered at \(\x\) with radius
\(\varepsilon>0\).
For a nonempty closed convex set \(\mathcal{C}\), the normal cone to \(\mathcal{C}\) at \(\x\in \mathcal{C}\) is
$
\mathcal{N}_{\mathcal{C}}(\x):=
\{\bm{v}:\langle \bm{v},\u-\x\rangle\le0,\ \forall \u\in \mathcal{C}\}$.
If \(\x\notin \mathcal{C}\), we set \(\mathcal{N}_{\mathcal{C}}(\x)=\emptyset\). 
For a proper closed convex function \(\varphi\),
\(\partial\varphi\) denotes its convex subdifferential; for convex-composite functions, we use \(\partial(h\circ c)(\x):=\nabla c(\x)\partial h(c(\x))\).

\section{Problem Setup and Local Multiplier Regularity}

In this section, we introduce the problem setup, key concepts, and standing assumptions that will serve as the basis for our subsequent analysis.
\subsection{Problem Setup}
\begin{assumption}[Basic problem setup]\label{ass:basic} Problem \eqref{eq:problem} satisfies the following conditions:
\begin{enumerate}[label=\normalfont(\roman*)]
\item 
For each $i\in\{0,\ldots,d\}$, the function $h_i:\R^m\rightarrow \R$  is convex and $L_h$-Lipschitz continuous:
\[
|h_i(\z)-h_i(\z')| \leq L_h\|\z-\z'\| \quad \forall \z,\z'\in \R^m;
\]
and $c_i:\R^n\rightarrow \R^m$ is twice continuously differentiable and has an $L_{c}$-Lipschitz Jacobian map:
\[
\|\nabla c_i(\x)-\nabla c_i(\x')\| \leq L_c\|\x-\x'\| 
\quad 
\forall \x, \x' \in \mathcal{X}.
\]
\item 
The set $\mathcal X\subseteq\R^n$ is nonempty, convex, and compact. Moreover, there exist $R_{\x}>0$ and $L>0$ such that
$\sum_{i=1}^d \max_{\x,\bar\x\in\mathcal X}|h_i(c_i(\bar\x)+\nabla c_i(\bar\x)^\top(\x-\bar\x))| \leq R_{\x}$ and
\[
|h_i(c_i(\x))-h_i(c_i(\x'))|\leq L\|\x-\x'\| \quad \forall i\in\{0,\ldots,d\},\ \forall \x,\x'\in\mathcal X. 
\]
Increasing $L$ if necessary, we assume throughout that
$L\ge\max\{1,L_hL_c\}$ and that, for every $i=0,1,\ldots,d$,
\begin{equation}\label{eq:gd_upper}
\sup_{\x,\bar\x\in\mathcal X}
\sup_{\s_i\in\partial h_i(c_i(\bar\x)+\nabla c_i(\bar\x)^\top(\x-\bar\x))}
\|\nabla c_i(\bar\x)\s_i\|\le L.
\end{equation}
\end{enumerate}
\end{assumption}

For simplicity, we denote 
\[
f(\x):=h_0(c_0(\x))\quad \text{and}\quad G(\x):=(h_1(c_1(\x)),\ldots,h_d(c_d(\x))). 
\]
We also write $f_{\min}:=\min_{\x\in\mathcal X}f(\x)$, which is finite by compactness.
For $\rho>0$, define the augmented Lagrangian function  $\mathcal{L}_\rho:\R^n\times\R^d\rightarrow \R$ as follows:
\[
\mathcal{L}_{\rho}(\x,\y):=f(\x)+\frac{\rho}{2}\sum_{i=1}^d \left[\left(\frac{y_i}{\rho}+G_i(\x)\right)_+^2- \frac{y_i^2}{\rho^2}\right].
\]
This is the standard Hestenes--Powell--Rockafellar augmented Lagrangian
associated with the inequality constraint \(G(\x)\le \bz\)
\cite{hestenes1969multiplier,powell1969method,rockafellar1973dual}. 

Problem \eqref{eq:problem} is equivalent
to
$
\min_{\x\in\mathcal X}\sup_{\y\in\mathbb R_+^d}
\mathcal L_\rho(\x,\y)$. The minimax representation here is taken over the dual cone $\mathbb R^d_+$. For algorithmic and analytical purposes, we introduce
an ancillary compact truncation of the dual domain $\mathcal{Y}\subseteq \mathbb R^d_+$. 

\begin{definition}[Auxiliary dual set]
\label{ass:compact}
The auxiliary dual set \(\mathcal Y\) is a projection-friendly compact
convex subset of \(\mathbb R^d_+\) with \(\bz\in\mathcal Y\). For
simplicity, we take
$\mathcal Y
:=
\{\y\in\mathbb R^d_+:\|\y\|_1\le R_{\y}\}.
$ 
\end{definition}

At this point, Definition \ref{ass:compact} only specifies the shape of the auxiliary dual truncation set. We do not assume a priori that any dual accumulation point or Lagrange multiplier associated with the algorithm lies in $\mathcal Y$. Later, \(R_{\y}\) will be chosen explicitly in terms of problem-dependent constants. 

\subsection{Local Multiplier Regularity}
Our primal-dual analysis relies on the following local regularity condition for the linearized composite constraints. It uniformly bounds the norm of nonnegative multipliers supported on nearly active constraints in terms of the associated first-order residual.

\begin{assumption}[Local regularity condition]
\label{ass:ucq}
For $\delta>0$, define  the near-feasible 
region $\mathcal R_{\delta}:=\{\x\in\mathcal X:\|G(\x)_+\|\le\delta\}$.
There exist constants $\delta_{\rm cq}>0$ and $\bar\sigma>0$ such that the following holds. 
For every $\x,\bar\x\in\mathcal X$ satisfying $\x\in\mathcal R_{\delta_{\rm cq}}$
and $\|\x-\bar\x\|\le\delta_{\rm cq}$, define the index set of nearly active constraints by
\[
\mathcal I_{\rm cq}(\bar\x,\x):=\left\{i\in[d]: h_i(c_i(\bar\x)+\nabla c_i(\bar\x)^\top(\x-\bar\x))\ge-\delta_{\rm cq}\right\}.
\]
For every choice of  $\s_i\in\partial h_i(c_i(\bar\x)+\nabla c_i(\bar\x)^\top(\x-\bar\x))$,
$i\in[d]$, set $\g_i:=\nabla c_i(\bar\x)\s_i$.
Then, every $\bm\mu\in\mathbb R_+^d $ satisfying $\mu_i=0$ for $i\notin\mathcal I_{\rm cq}(\bar\x,\x)$ obeys
\[
\bar\sigma\|\bm\mu\|
\le
\dist\left(\bz,\ \sum_{i=1}^d\mu_i\g_i+\mathcal N_{\mathcal X}(\x)\right).
\]
\end{assumption} 
\begin{lemma}[Relation to classical smooth regularity]
\label{lem:local-cq-classical}
Let $\mathcal X$ be nonempty, compact, and convex, and let
$G_1,\ldots,G_d$ be continuously differentiable on a neighborhood
of $\mathcal X$.
Set $h_i(u)=u$ and $c_i=G_i$ for $i\in[d]$.
Then the following statements hold.
\begin{enumerate}[label=\normalfont(\roman*)]

\item \textup{(MFCQ \cite[Theorem~6.14 and Example~6.40]{rockafellar2009variational})}
Assumption~\ref{ass:ucq} holds for some
$\delta_{\rm cq},\bar\sigma>0$ if and only if the following
implication holds at every feasible point
$\x^\star\in\{\x\in\mathcal X:G(\x)\le\bz\}$
for every $\bm\mu\in\mathbb R^d$:
\begin{equation}\label{eq:MFCQ_imply}
\left.
\begin{aligned}
&\bz\in
  \sum_{i=1}^d\mu_i\nabla G_i(\x^\star)
  +\mathcal N_{\mathcal X}(\x^\star),\\
&\bm\mu\ge\bz,\quad
  \mu_iG_i(\x^\star)=0,\quad i\in[d]
\end{aligned}
\right\}
\quad\Longrightarrow\quad
\bm\mu=\bz.
\end{equation}
\item \textup{(Slater's condition)}
Suppose that every $G_i$ is convex on $\mathcal X$.
If there exists $\x'\in\mathcal X$ such that
$G_i(\x')<0$ for all $i\in[d]$, then
Assumption~\ref{ass:ucq} holds for some
$\delta_{\rm cq},\bar\sigma>0$.
\item \textup{(P\L{} inequality)}
Suppose that Assumption~\ref{ass:ucq} holds with constants
$\delta_{\rm cq},\bar\sigma>0$.
Then the constraint violation function
$\psi(\x):=\frac12\|G(\x)_+\|^2$ satisfies
\begin{equation}\label{eq:local-feasibility-PL}
\frac12
\dist^2\!\left(
\bz,\nabla\psi(\x)+\mathcal N_{\mathcal X}(\x)
\right)
\ge
\bar\sigma^2\psi(\x),
\qquad
\forall \x\in\mathcal R_{\delta_{\rm cq}}.
\end{equation}

\end{enumerate}
\end{lemma}

\begin{remark}
Lemma~\ref{lem:local-cq-classical} shows that, for smooth
constraints, Assumption~\ref{ass:ucq} is a locally uniform
formulation of MFCQ relative to $\mathcal X$.
In particular, for convex constraint functions,
the strict Slater condition used in, e.g.,
\cite[Assumption~2(1)]{pu2024smoothed} suffices.
The assumption further yields the local  P\L{}
inequality \eqref{eq:local-feasibility-PL}, even when
the constraints are nonconvex.
This provides a local analogue for inequality constraints
of the feasibility regularity conditions used for equality
constraints, which are imposed throughout the domain in
\cite[Assumption~3]{li2024stochastic}
or along the generated iterates in
\cite[Assumption~A5]{alacaoglu2024complexity}.
Our analysis first establishes near-feasibility of the
selected good iterates without relying on
Assumption~\ref{ass:ucq}.
We invoke this local condition only afterward to bound
the associated multipliers and obtain a KKT certificate;
see Section~\ref{subsec:good_iterates} for more details.
\end{remark}

The following facility location model satisfies
Assumption~\ref{ass:ucq}, although its constraint violation
function does not satisfy a global P\L{} inequality. 
\begin{example}[Obnoxious facility location~\cite{drezner2019planar}]
We consider locating two facilities, such as landfills or polluting
factories, away from communities. The objective maximizes the minimum
squared distance between facilities and communities, while a minimum
separation between the facilities avoids cumulative
nuisance~\cite{drezner2019planar}.
Following the numerical setting in that paper, we take the location
region to be $[0,1]^2$ and the minimum separation to be $\tfrac{1}{2}$.
Given community locations $\bm{a
}_1,\ldots,\bm{a}_N\in[0,1]^2$, eliminating
the auxiliary objective variable gives
\[
\begin{aligned}
\min_{\x=(\bm{u},\bm{v})\in\cX=[0,1]^4}\quad&
\max_{1\le i\le N}
\bigl\{-\|\bm{u}-\bm{a}_i\|^2,-\|\bm{v}-\bm{a}_i\|^2\bigr\}\\
\mathrm{s.t.}\quad&
G(\x):=\tfrac14-\|\bm{u}-\bm{v}\|^2\le0.
\end{aligned}
\]
\end{example}
The maximum outer function is convex and Lipschitz, and the inner
functions and $G$ are quadratic. Together with compactness of
$\cX$, these properties give Assumption~\ref{ass:basic}.
The opposite corners $\bm{u}=(0,0)$ and $\bm{v}=(1,1)$ provide a feasible
initial point.

We verify Assumption~\ref{ass:ucq} using Lemma~\ref{lem:local-cq-classical}(i).
It suffices to show that MFCQ relative to $\cX$ holds at every
feasible point. Let $\bm{x}=(\bm{u},\bm{v})$ be feasible and
suppose that $\mu\ge0$ satisfies $\bm{0}\in\mu\nabla G(\bm{x})+\mathcal{N}_{\cX}(\bm{x})$ and $\mu G(\bm{x})=0$.
If $G(\bm{x})<0$, complementarity gives $\mu=0$.
Otherwise, $\|\bm{u}-\bm{v}\|=\tfrac12$.
Suppose for contradiction that $\mu>0$.
Since $\mathcal{N}_{\cX}(\bm{x})$ is a cone, we obtain
\[
-\nabla G(\bm{x})
=\bigl(2(\bm{u}-\bm{v}),-2(\bm{u}-\bm{v})\bigr)
\in \mathcal{N}_{\cX}(\bm{x}).
\]
Choose a coordinate $k$ with $u_k>v_k$, interchanging
$\bm{u}$ and $\bm{v}$ if necessary.
For the box $\cX=[0,1]^4$, a positive normal component requires
the corresponding coordinate to equal $1$, while a negative
normal component requires it to equal $0$.
It therefore forces $u_k=1$ and $v_k=0$.
This gives $\|\bm{u}-\bm{v}\|\ge|u_k-v_k|=1$,
contradicting $\|\bm{u}-\bm{v}\|=\tfrac12$.
Hence $\mu=0$, proving the required MFCQ condition.

In contrast, the squared constraint violation
$\psi(\bm{x})=\tfrac12 G(\bm{x})_+^2$ does not satisfy a global
P\L{} inequality on $\cX$.
Consider the coincident configuration
$\bm{x}_c=(\tfrac12,\tfrac12,\tfrac12,\tfrac12)
\in\operatorname{int}\cX$.
At this point, $G(\bm{x}_c)=\tfrac14$,
$\nabla G(\bm{x}_c)=\bm{0}$, and
$\mathcal{N}_{\cX}(\bm{x}_c)=\{\bm{0}\}$.
Hence, we have $\psi(\bm{x}_c)=\tfrac1{32}>0$ and $\operatorname{dist}\bigl(
\bm{0},\nabla\psi(\bm{x}_c)+\mathcal{N}_{\cX}(\bm{x}_c)
\bigr)=0$.
Thus no positive constant can make the P\L{} inequality in
Lemma~\ref{lem:local-cq-classical}(iii) hold throughout $\cX$.

\subsection{Stationarity Measure}
We next define the stationarity measures used in the paper. Before doing so,
we record some  basic properties of the smoothed augmented Lagrangian. The proof
is deferred to Appendix~\ref{proof:lemma_lrho}.

\begin{lemma}\label{lem: weakcvx}
The function \(\mathcal{L}_{\rho}(\cdot, \y)\) is
\(L_{\rho}:=L(1+R_{\y}+\rho R_{\x})\)-weakly convex on
\(\mathcal{X}\) with the nonsmooth composite structure for any
\(\y\in\mathcal{Y}\subseteq \R^d_+\), and
the function \(\mathcal L_\rho(\x,\cdot)\) is concave and
continuously differentiable on \(\mathcal Y\).
\end{lemma}

Given $r_{\x},r_{\y}\ge0$, we define the regularized augmented
Lagrangian $F:\R^n\times \R^d\times \R^n\rightarrow\R$ and the dual function $d:\R^d\times \R^n\rightarrow \R$ as
\[
F(\x,\y,\z) := \mathcal{L}_{\rho}(\x,\y)+\frac{r_{\x}}{2}\|\x-\z\|^2-\frac{r_{\y}}{2}\|\y\|^2\quad \text{and}\quad d(\y,\z):=\min_{\x\in\mathcal{X}} F(\x,\y,\z),
\]
where we assume
$r_{\x}> L_\rho$. This regularized potential function follows the standard smoothing/regularization technique used in nonconvex minimax analysis, where the primal proximal term stabilizes the minimization over $\x$, while the negative quadratic term in $\y$ improves the concavity and regularity of the dual side; see, e.g., \cite{zhang2020single,li2025nonsmooth,li2026smoothing}. Here, \(F(\cdot,\y,\z)\) is strongly convex on \(\mathcal X\) for
 every \((\y,\z)\in\mathcal Y\times\mathbb R^n\). Hence, we can use the following stationarity measure. 
\begin{definition}
\label{defi:primal-dual}
The pair $(\x,\y)\in\mathcal{X}\times \mathcal{Y}$ is an $\epsilon$-KKT point if 
\[
\|\nabla_{\z}d(\y,\x)\|\le \epsilon,
\quad \text{and} \quad 
\dist\left(
\bz,
-\nabla_{\y}\mathcal L_\rho(\x,\y)
+
\mathcal{N}_{\mathbb R^d_+}(\y)\right)
\le \epsilon.
\]
\end{definition}

\begin{remark}\label{rmk:kkt-recovery}
 Although the pair \((\x,\y)\) is evaluated over \(\mathcal X\times\mathcal Y\), the dual stationarity residual is defined with respect to the original dual cone \(\R^d_+\), rather than the auxiliary truncation set \(\mathcal Y\). The first condition in Definition~\ref{defi:primal-dual} is a composite primal stationarity measure in the spirit of \cite[Definition~1]{li2025nonsmooth}.
The second condition gives $\nabla_{\y}\mathcal L_{\rho}(\x,\y)\in \mathcal N_{\R^d_+}(\y)$, when $\epsilon=0$.  This is equivalent componentwise to \[ G_i(\x)\le 0,\quad y_i\ge 0,\quad y_iG_i(\x)=0, \qquad \forall i\in[d]. \] Thus the second condition recovers primal feasibility, dual feasibility, and complementary slackness. Together with the primal stationarity condition, Definition~\ref{defi:primal-dual} gives the full KKT system of problem~\eqref{eq:problem} when $\epsilon=0$.
\end{remark}

In the remaining part of the paper, we always assume that Assumptions \ref{ass:basic} and \ref{ass:ucq} hold, and $\mathcal{Y}$ is given by Definition \ref{ass:compact}.

\section{Smoothed Prox-Linear Augmented Lagrangian Method}
\label{sec:main}
In this section, we present the smoothed prox-linear
augmented Lagrangian method. 
The composite functions
\(h_i\circ c_i\), \(i=0,1,\ldots,d\), are generally nonsmooth and
nonconvex. Hence their gradients may not be available. We therefore
linearize only the smooth inner maps \(c_i\) while keeping the convex
outer functions \(h_i\) as in the prox-linear scheme.
For each \(k\ge0\), define the linearized inner maps
$
c_i^k(\x)
:=
c_i(\x^k)+\nabla c_i(\x^k)^\top(\x-\x^k),
$ for all \(i=0,1,\ldots,d\). 
This leads to the following prox-linear primal
update:
    \[
    \x^{k+1}
    =
    \mathop{\argmin}_{\x\in\mathcal{X}}
    \left\{
    F_{\x^k,\lambda}(\x,\y^k)
    +\frac{r_{\x}}{2}\|\x-\z^{k}\|^2
    \right\},
    \]
where
\begin{equation}
\label{eq:prox_linear_model}
\begin{aligned}
F_{\x^k,\lambda}(\x,\y^k)
:=
&\ h_0(c_0^k(\x))
+\frac{1}{2\lambda}\|\x-\x^k\|^2+\frac{\rho}{2}\sum_{i=1}^d
\left[
\left(
\frac{y^k_i}{\rho}
+h_i(c_i^k(\x))
\right)_+^2
-
\frac{(y^k_i)^2}{\rho^2}
\right].
\end{aligned}
\end{equation}
Here, \(F_{\x^k,\lambda}(\cdot,\y^k)\) is a convex prox-linear model of
\(\mathcal L_\rho(\cdot,\y^k)\) at \(\x^k\). 
For the dual update, we use the closed-form expression
\[
\nabla_{\y}F(\x,\y,\z)
=
\nabla_{\y}\mathcal L_\rho(\x,\y)-r_{\y}\y
=
\max\left\{G(\x),-\frac{\y}{\rho}\right\}
-r_{\y}\y,
\]
where the maximum is taken componentwise. Thus, the dual variable is
updated by a projected gradient ascent step over the ancillary set
\(\mathcal Y\), and the auxiliary variable \(\z\) is updated by a
standard gradient descent step on $F(\x^{k+1},\y^{k+1},\cdot)$. The full method is presented in
Algorithm~\ref{alg}.

\begin{algorithm}[H]
  \caption{Smoothed Prox-Linear ALM}
  \label{alg}
  \DontPrintSemicolon
  \SetKwInOut{Input}{Input}
  \Input{Initial point $\x^0\in\{\x\in\mathcal{X}:h_i(c_i(\x))\leq 0,\ \forall i\in[d]\}$, $\y^0=\bz$, $\z^0=\x^0$, and parameters $\rho>0$, $r_{\x}>L_\rho$, $r_{\y}\ge0$, $\lambda>0$, $\alpha>0$, $\beta\in(0,1)$.}
  \For{$k=0,1,\ldots$}{
    $\displaystyle \x^{k+1}:= \mathop{\argmin}_{\x \in \X}\left\{ F_{\x^k,\lambda}(\x,\y^k) +\frac{r_{\x}}{2}\|\x-\z^{k}\|^2 \right\}$\;
    $\displaystyle \y^{k+1}:= \proj_{\mathcal{Y}}\left(\y^k+\alpha\left(\max\left\{G(\x^{k+1}),-\frac{\y^k}{\rho}\right\}-r_{\y}\y^k\right)\right)$\;
    $\displaystyle \z^{k+1}:=\z^k+\beta(\x^{k+1}-\z^k)$\;
  }
\end{algorithm}
\begin{remark}[Primal subproblem]
The primal step is a strongly convex composite subproblem. If projection onto
$\mathcal X$ and the relevant proximal or epigraphical operations for $h_i$
are tractable, it can be solved by standard convex splitting
methods~\cite{boyd2011distributed,chambolle2011first} after introducing
$u_i=c_i^k(\x)$. Our rates count outer iterations and assume exact solutions
of the primal subproblems; an inexact analysis is beyond the present scope.
\end{remark}

\section{Main Results and Proof Overview}
\label{sec:roadmap}
We first provide a roadmap of the convergence analysis and state the main technical results leading to the iteration complexity guarantee for Algorithm~\ref{alg}. The detailed proofs are deferred to the subsequent sections and appendices. Our starting point is to view Algorithm~\ref{alg} as a nonconvex minimax method over the bounded domain \(\mathcal X\times\mathcal Y\), where \(\mathcal Y\) is the auxiliary dual truncation set introduced in Definition~\ref{ass:compact}. Once \(\mathcal Y\) is fixed and compact, the basic Lyapunov descent mechanism follows the standard primal descent, dual ascent, and proximal descent framework for nonconvex-concave minimax optimization \cite{zhang2020single,li2025nonsmooth,li2026smoothing}. 
For completeness, we recall the corresponding bounded-domain descent estimate
specialized to the present prox-linear augmented Lagrangian model.

To streamline the statements, we introduce several auxiliary value functions
and solution maps. For \((\y,\z)\in\mathcal Y\times\mathcal X\), define
$
\x(\y,\z)
:=
\operatorname*{argmin}_{\x\in\mathcal X}F(\x,\y,\z)$,
which is unique because
\(F(\cdot,\y,\z)\) is \((r_{\x}-L_\rho)\)-strongly convex. Next, define
\[
p(\z)
:=
\max_{\y\in\mathcal Y}d(\y,\z),
\quad
\text{and} \quad
\mathcal Y(\z)
:=
\operatorname*{argmax}_{\y\in\mathcal Y}d(\y,\z).
\]
Given any selection  \(\y(\z)\in\mathcal Y(\z)\), we have 
$
\x^\star(\z)=\x(\y(\z),\z)$ for the corresponding primal minimizer.
For a current $\y\in\mathcal Y$, we also define the projected dual ascent map
\[
\y_+(\z)
:=
\proj_{\mathcal Y}
\left(\y+\alpha\nabla_{\y}d(\y,\z)\right).
\]
With this notation, define the Lyapunov function $\Phi:\R^n\times\R^d \times\R^n \rightarrow\R$ as:
\[
\Phi(\x,\y,\z):= \underbrace{F(\x,\y,\z)-d(\y,\z)}_{\textnormal{Primal Descent}} + \underbrace{p(\z)-d(\y,\z)}_{\textnormal{Dual Ascent}} + \underbrace{p(\z)}_{\textnormal{Proximal Descent}}. 
\]
The three terms correspond, respectively, to the primal descent step, the projected dual ascent step, and the proximal descent of the auxiliary variable \(\z\).

Let \(L_G>0\) be a Lipschitz constant of the constraint mapping \(G\) on
\(\mathcal X\), i.e.,
\[
\|G(\x)-G(\x')\|
\le
L_G\|\x-\x'\|,
\qquad
\forall \x,\x'\in\mathcal X.
\]
Under the coordinatewise Lipschitz bound in Assumption~\ref{ass:basic},
one may take \(L_G=\sqrt d L\). For later use, define
\begin{equation}\label{eq:constants}
\sigma_1
:=
\frac{r_{\x}}{r_{\x}-L_\rho},
\qquad
\sigma_2
:=
\frac{L_G}{r_{\x}-L_\rho},
\quad \text{and} \quad
L_d
:=
L_G\sigma_2+r_{\y}+\frac{1}{\rho}.
\end{equation}

The first building block is a standard basic Lyapunov descent
estimate. Its proof follows the standard nonconvex-concave minimax Lyapunov analysis;
we include the proof in
Appendix~\ref{sec:suff_decrease} to make the constants explicit.
\begin{proposition}[Basic descent of \(\Phi\)]
\label{prop:decrease}
Suppose that 
$
r_{\x}\ge \max\{3L_\rho, L_\rho+2L_G\}$. 
Let $0<\lambda\le \frac{1}{L_G}$,
\[
  0<\alpha
\le
\min\left\{
\frac{1}{2(L_G\zeta^2+L_d)-r_{\y}},
\frac{1}{8L_G\zeta^2}
\right\},\quad 0<\beta
\le
\frac{1}{28}
\min\left\{
1,\frac{2}{\alpha r_{\x}\sigma_2^2}
\right\},\]
where $
\zeta
:=
\left(
1+\frac{2(\lambda^{-1}+L_\rho)}{r_{\x}-L_\rho}
\right)
\left(
\sqrt{\frac{2L_\rho}{\lambda^{-1}+L_\rho}}+1
\right)$. 
Then for any \(k\ge0\), we have
\[
\begin{aligned}
\Phi^k-\Phi^{k+1}
\ge\ &
\frac{7}{16\lambda}\|\x^{k}-\x^{k+1}\|^{2}
+\frac{1}{8\alpha}\|\y^{k}-\y_{+}^{k}(\z^{k})\|^2
+\frac{4r_{\x}\beta}{7}\|\z^{k}-\x^{k+1}\|^{2} \\
&-
28r_{\x}\beta
\|\x(\y(\z^{k}),\z^{k})
-\x(\y_+^{k}(\z^k), \z^{k})\|^{2},
\end{aligned}
\]
where \(\Phi^k:=\Phi(\x^k,\y^k,\z^k)\).
\end{proposition}

Two main technical issues remain after Proposition~\ref{prop:decrease}.
First, we need to bound its negative sensitivity term in terms of
$\|\y^k-\y_+^k(\z^k)\|^2$. For $r_{\y}>0$, strong concavity gives a
global dual error bound, and balancing the associated regularization bias
yields the $\mathcal O(K^{-1/3})$ rate. For $r_{\y}=0$, we instead
establish a local dual error bound under
Assumption~\ref{ass:struc-regularity}, leading to the sharper
$\mathcal O(K^{-1/2})$ rate.

Second, the descent estimate is formulated over the truncated dual set
$\mathcal Y$, whereas the KKT system for problem~\eqref{eq:problem} uses
the original dual cone $\mathbb R_+^d$. The key technical difficulty is
therefore to identify sufficiently many near-feasible, nearly stationary
iterates and prove that the artificial dual truncation is inactive at these
iterates. The local nature of Assumption~\ref{ass:ucq} is essential here:
Rather than imposing multiplier regularity over the entire
constraint-violation region, we first use counting arguments to identify such
iterates and then invoke Assumption~\ref{ass:ucq} only on the near-feasible
region to bound the associated multipliers. Choosing $R_{\y}$ above this bound
makes the truncation inactive at these iterates and converts the resulting
residual bounds into KKT residual bounds for problem~\eqref{eq:problem}.

\subsection{Dual Error Bounds and Sufficient Decrease}

We now address the first technical question identified above: absorbing the negative sensitivity term in Proposition~\ref{prop:decrease}. This is done through dual error bounds that control $ \|\x(\y(\z^k),\z^k)-\x(\y_+^k(\z^k),\z^k)\|$ by the projected dual residual \(\|\y^k-\y_+^k(\z^k)\|\). We state two such bounds: a global dual error bound for the regularized case \(r_{\y}>0\), and a local dual error bound for the unregularized case \(r_{\y}=0\).

\begin{proposition}[Global  dual error bound] 
\label{prop:dual_eb_KL}
Let \(r_{\y}>0\). Then for any
\((\y,\z)\in\mathcal{Y}\times \R^{n}\),
\begin{equation}\label{eq:dualeb1}
\|\x(\y_+(\z),\z)-\x(\y(\z),\z)\|
\leq
\omega_1\|\y_+(\z)-\y\|,
\end{equation}
where $\omega_1:=\frac{1+\alpha L_d}{\alpha\sqrt{r_{\y}(r_{\x}-L_\rho)}}$.
\end{proposition}
\begin{proof}
Fix \(\z\). Since \(\mathcal L_\rho(\cdot,\y)\) is \(L_\rho\)-weakly
convex, \(F(\cdot,\y,\z)\) is \((r_{\x}-L_\rho)\)-strongly convex. Since
\(r_{\y}>0\), \(F(\x,\cdot,\z)\) is \(r_{\y}\)-strongly concave, and
therefore the dual problem satisfies the KL inequality with
exponent \(\frac{1}{2}\).
Applying
\cite[Proposition~3]{li2025nonsmooth} with $\theta=\frac12$, $\mu=\sqrt{2r_{\y}}$
yields \eqref{eq:dualeb1}.
\end{proof}

The case \(r_{\y}=0\) requires a local argument. We use the following
additional assumptions only for the local dual error bound established in Proposition \ref{prop:dual_eb_local}. 

\begin{assumption}
\label{ass:struc-regularity}
For each \(i\in\{0,\ldots,d\}\), there exist a finite nonempty set
\(\mathcal A_i\) and
\((\a_{i\ell},b_{i\ell})\in\mathbb R^m\times\mathbb R\) such that
$h_i(\w)=\max_{\ell\in\mathcal A_i}
\{\langle\a_{i\ell},\w\rangle+b_{i\ell}\}.
$
Define
\(\phi_{i\ell}(\x):=\langle\a_{i\ell},c_i(\x)\rangle+b_{i\ell}\)
for \(\ell\in\mathcal A_i\), and set
\(\mathcal A_i(\x):=\argmax_{\ell\in\mathcal A_i}\phi_{i\ell}(\x)\).
Also define
\[
\mathcal S
:=
\left\{
(\x,\y)\in\mathcal X\times\mathcal Y:
\y\in\mathcal Y(\x),\
\x=\x(\y,\x)
\right\}. 
\]
For \(\x\in\mathcal X\), let
\(\mathcal I(\x):=\{i\in[d]:G_i(\x)=0\}\) and
\(\mathcal I_+(\x):=\{i\in[d]:G_i(\x)\geq0\}\).
There exists \(\tilde\sigma>0\) such that every
\((\x^\star,\y^\star)\in\mathcal S\) satisfies:
\begin{enumerate}[label=\normalfont(\roman*)]

\item 
\(y_i^\star>0\) for every \(i\in\mathcal I(\x^\star)\);
\item for every \(\ell_0\in\mathcal A_0(\x^\star)\), the matrix
\[
M(\x^\star,\ell_0)
:=
\Bigl[
\bigl(
\nabla(\phi_{0\ell}-\phi_{0\ell_0})(\x^\star)
\bigr)_{\ell\in\mathcal A_0(\x^\star)\setminus\{\ell_0\}}
,
\bigl(
\nabla\phi_{i\ell}(\x^\star)
\bigr)_{\substack{
i\in\mathcal I_+(\x^\star),\,
\ell\in\mathcal A_i(\x^\star)}}
\Bigr]
\]
satisfies
$M(\x^\star,\ell_0)^\top M(\x^\star,\ell_0)\succeq \widetilde{\sigma}^{\,2} \bm{I}$
;

\item  
$
\x^\star\in\operatorname{int}(\mathcal X)$.
\end{enumerate}
\end{assumption}

Under Assumption \ref{ass:struc-regularity} we can verify that \(\mathcal S\) coincides with KKT pairs whose multipliers lie in \(\mathcal Y\).
\begin{lemma}
\label{lem:S-eq-KKT}
Let \(r_{\y}=0\).   Suppose that Assumption~\ref{ass:struc-regularity}
holds and
$
R_{\y}>\sqrt{N_{\cA}}L/\tilde\sigma$ with $N_{\cA}:=\sum_{i=1}^d |\cA_i|$.
Then every \((\x^\star,\y^\star)\in\mathcal S\) satisfies
$
G(\x^\star)\le\bz$,
$\|\y^\star\|_1<
R_{\y}$,
and \((\x^\star,\y^\star)\) is a KKT pair of~\eqref{eq:problem}.
Conversely, every KKT pair \((\x^\star,\y^\star)\) of~\eqref{eq:problem}
with \(\y^\star\in\mathcal Y\) belongs to \(\mathcal S\). 
\end{lemma}

\begin{remark}
By Lemma~\ref{lem:S-eq-KKT}, parts (i) and (ii) are strict
complementarity and a piecewise-linear analogue of LICQ on the relevant KKT
set. Part (iii) only simplifies the local sensitivity analysis for $r_{\y}=0$
and is not needed to run Algorithm~\ref{alg}. If the boundary of
$\mathcal X$ can be active, the regularity condition must additionally
incorporate $\mathcal N_{\mathcal X}$, as in \cite[Assumption~2]{pu2024smoothed}.
In the smooth case, where $h_i(u)=u$ for $i=0,\ldots,d$,
Lemma~\ref{lem:local-cq-classical} shows that
Assumption~\ref{ass:ucq} is equivalent to MFCQ relative to
$\mathcal X$ at every feasible point.
The LICQ requirement in
Assumption~\ref{ass:struc-regularity}, however,
is imposed uniformly only at the KKT pairs in $\mathcal S$.
LICQ implies MFCQ at these interior points, but the converse
need not hold.
Moreover, LICQ at these KKT pairs does not ensure MFCQ
at other feasible points.
Thus, neither assumption implies the other in general.
In our analysis, Assumption~\ref{ass:ucq} bounds the multipliers
at the good iterates, whereas
Assumption~\ref{ass:struc-regularity} provides the local
solution stability needed to establish the dual error bound
in Proposition~\ref{prop:dual_eb_local}.
\end{remark}

\begin{proposition}[Local  dual error bound]
\label{prop:dual_eb_local}
Let \(r_{\y}=0\) and \((\x^\star,\y^\star)\in\mathcal{S}\). 
 Suppose that Assumption~\ref{ass:struc-regularity} holds, and $R_{\y}>\sqrt{N_{\cA}}L/\tilde \sigma. $
Then there exist constants \(\epsilon>0\) and \(\omega_2>0\),
independent of \(\z\) in a sufficiently small neighborhood of
\(\x^\star\), and a local selection \(\y(\z)\in\mathcal Y(\z)\) with
\(\y(\z)\to\y^\star\) as \(\z\to\x^\star\), such that, for any
\(\y\in\mathcal Y\) satisfying
$\y,\y_+(\z)\in\mathbb B_{\epsilon}(\y(\z))$,
we have
\[
    \|\x(\y(\z),\z)-\x(\y_+(\z),\z)\|
\le
\omega_2\|\y_+(\z)-\y\|.
\]
\end{proposition}

To use the local dual error bound in Proposition~\ref{prop:dual_eb_local}
to handle the negative term in Proposition \ref{prop:decrease}, we must first ensure that the
relevant iterates lie in the neighborhood where the local dual error bound is valid.
The next two auxiliary lemmas provide this bridge. The first lemma gives a
global nonhomogeneous dual error bound. Although weaker than the local
dual error bound, it is sufficient to show that if the negative term in
the basic descent estimate is large, then the associated algorithmic
residuals must be small. Its proof is an adaptation of the argument in
\cite[Lemma~8]{li2025nonsmooth}. The second lemma then uses these small
residuals to place the iterates inside the local regime required by
Proposition~\ref{prop:dual_eb_local}.

\begin{lemma}
\label{lemma-dual-bd}
 Let  $\kappa
:=
\frac{\sqrt{2}
(\alpha^{-1}+L_d)R_{\y}
}{
r_{\x}-L_\rho
}
$. For any
\((\y,\z)\in\mathcal Y\times\R^n\), 
\[
\|\x(\y(\z),\z)-\x(\y_+(\z),\z)\|^{2}
\leq
\kappa\|\y-\y_+(\z)\|. 
\]
\end{lemma}

\begin{lemma}
\label{lem:local_entry}
Let \(r_{\y}=0\). Suppose that the assumptions of
Proposition~\ref{prop:dual_eb_local} hold.
Then there exists
\(\bar\delta>0\) such that if
\[
\max\{
\|\x^k-\x^{k+1}\|,
\|\y^k-\y_+^k(\z^k)\|,
\|\z^k-\x^{k+1}\|
\}
\le
\bar\delta,
\]
we have
$
\y^k,\ \y_+^k(\z^k)
\in
\mathbb B_{\epsilon}(\y(\z^k))$.
\end{lemma}

Combining the basic descent estimate with the appropriate dual error
bound yields the following sufficient decrease property.

\begin{proposition}[Sufficient decrease property]
\label{prop:suff-decrease}
Suppose that the assumptions of Proposition~\ref{prop:decrease} hold. Assume that one of the following two cases holds.
\begin{enumerate}[label=\normalfont(\roman*)]
\item \textnormal{(\(r_{\y}>0\)).}
Suppose that
$
\beta
\le
\frac{1}{448 r_{\x}\omega_1^2\alpha}.
$ 
\item \textnormal{(\(r_{\y}=0\)).}
Suppose that Assumption \ref{ass:struc-regularity} holds and $ 
R_{\y}>
\sqrt{N_{\cA}} L/\tilde\sigma
$. 
Define
$
\varrho_1:=448r_{\x}\alpha\kappa,\allowbreak
\varrho_2:=\sqrt{98\kappa\varrho_1},
$   and $
\varrho_3:=\sqrt{128r_{\x}\lambda\kappa\varrho_1}.
$ 
Assume that
$
\beta
\le
\min\{
\frac{\bar\delta}{\varrho_1},
\frac{\bar\delta^2}{\varrho_2^2},
\frac{\bar\delta}{\varrho_3},
\frac{1}{448r_{\x}\omega_2^2\alpha}
\},
$ 
where $\bar\delta>0$ is the constant from Lemma~\ref{lem:local_entry}.
\end{enumerate}
Then for any \(k\ge0\), we have
\begin{equation}\label{eq:suff-decrease}
\begin{aligned}
\Phi^k-\Phi^{k+1}
\ge\ &
c_\beta(\|\x^{k}-\x^{k+1}\|^{2}+\|\y^{k}-\y_{+}^{k}(\z^{k})\|^2+\|\z^{k}-\x^{k+1}\|^{2}),
\end{aligned}
\end{equation}
where  $
c_\beta
:=
\min\{
\frac{3}{16\lambda},
\frac{1}{16\alpha},
\frac{2r_{\x}\beta}{7}
\}$. 
\end{proposition}

\subsection{Good Iterates: Near-Feasibility and Iteration Complexity}
\label{subsec:good_iterates}

Next, we identify a large set of good iterations. These iterations have
two properties: The primal point is nearly feasible, and the artificial
boundary of \(\mathcal Y\) is inactive at both \(\y^k\) and
\(\y^{k+1}\).

\begin{proposition}[Frequency of near-feasible iterates]
\label{prop:primalbd}
Suppose that the assumptions of Proposition~\ref{prop:suff-decrease}
hold with
$
 r_{\y}\leq \frac{1}{\alpha}$.
Then, for any fixed \(K\in\mathbb N\) and any \(\delta>0\),
\[
\#\left\{
0\le k\le K-1:
\|G(\x^{k+1})_+\|>\delta
\right\}
\le
\frac{2K(\Phi^0-f_{\min})}{\rho\delta^2}.
\]
Equivalently,
$
\|G(\x^{k+1})_+\|\le \delta$
holds for at least
$
K-
\frac{2K(\Phi^0-f_{\min})}{\rho\delta^2}$
indices \(k\).
\end{proposition}

\begin{remark} \label{rmk:near-feasible-no-cq} 
The near-feasibility estimate is independent of the local regularity condition
in Assumption~\ref{ass:ucq}. It follows from the penalty term and the Lyapunov
counting argument, before any constraint qualification is invoked. The local
regularity condition is used only afterward: Once a selected good iterate has
entered the region \(\mathcal R_{\delta_{\rm cq}}\), the primal optimality
condition together with Assumption~\ref{ass:ucq} yields a uniform bound on the
associated multiplier vector. This bound is then used in
Proposition~\ref{prop:dual-bd-ybd-linearized} to choose \(R_{\y}\) so that 
the artificial dual truncation is inactive at the
selected good iterates.
\end{remark}

\begin{proposition}[Inactive artificial dual truncation]
\label{prop:dual-bd-ybd-linearized}
Suppose that the assumptions of
Proposition~\ref{prop:primalbd} hold. Fix \(\xi>0\) and \(\delta\in(0,\delta_{\rm cq}]\). Suppose that
\(\rho\ge\alpha\) and
\(\sqrt{\xi/\alpha}+L(2+\zeta)\sqrt{\lambda\xi}
\le\delta_{\rm cq}\).
Define
\begin{equation}\label{eq:rho-K-good}
N_K:=K-
\frac{7(\Phi^0-f_{\min})}{2\beta\xi}
-
\frac{2K(\Phi^0-f_{\min})}{\rho\delta^2},
\end{equation}
and assume that \(N_K>0\).
If
\begin{equation}\label{eq:ryrequire}
R_{\y}>
\rho d\delta_{\rm cq}
+\frac{\sqrt d\,(L+(\sqrt{r_{\x}}+\sqrt{\lambda^{-1}})\sqrt\xi)}{\bar\sigma}
+\alpha\sqrt d\,\delta,
\end{equation}
then
$
\|\y^k\|_1<R_{\y}$ and $\|\y^{k+1}\|_1<R_{\y}$
hold for at least
$N_K$
indices \(k\). 
\end{proposition}
\begin{remark}
The initialization in Algorithm~\ref{alg} yields an upper 
bound on $\Phi^0$ that is independent of $\rho$ and $R_{\y}$. Since $\x^0$ is 
feasible and $\y^0 = \bz$, $\z^0 = \x^0$, we have 
$F(\x^0, \y^0, \z^0) = \mathcal{L}_\rho(\x^0, \bz) = f(\x^0)$. Moreover, feasibility 
of $\x^0$ implies $\mathcal{L}_\rho(\x^0, \y) \le f(\x^0)$ for every $\y \in \mathcal{Y}$. 
Hence
\[
p(\z^0) = \max_{\y \in \mathcal{Y}} d(\y, \x^0) 
\le \max_{\y \in \mathcal{Y}} \mathcal{L}_\rho(\x^0, \y) \le f(\x^0),
\quad \text{and} \quad
d(\y^0, \z^0) \ge f_{\min}.
\]
Therefore $\Phi^0 = F(\x^0,\y^0,\z^0) + 2\bigl(p(\z^0) - d(\y^0,\z^0)\bigr) \le 3f(\x^0) - 2f_{\min}$. 
In particular,
$\Phi^0-f_{\min}\le 3(f(\x^0)-f_{\min})$ is independent of $\rho$ and
$R_{\y}$. Consequently, all parameter requirements are jointly feasible:
choose $\rho$ sufficiently large and $\beta\xi$ so that $N_K>0$ and, for
the rate claims, $N_K\ge \nu K$ for some constant $\nu>0$ independent of
$K$; then take $R_{\y}=\Theta(\rho)$ and
$r_{\x}=\Theta(\rho^2)$ with sufficiently large constants.
\end{remark}

We now state the main iteration complexity result.

\begin{theorem}
\label{thm:general}
Suppose that the assumptions of Proposition
\ref{prop:dual-bd-ybd-linearized} hold. 
Fix
\(K\in\mathbb N\), \(\delta\in(0,\delta_{\rm cq}]\), and \(\xi>0\).
Then there exists an index \(k\in\{0,\ldots,K-1\}\) such that
$
(\x^{k+1},\y^{k+1})$
is an \(\mathcal O(\varepsilon_K)\)-KKT point, 
where
$
\varepsilon_K
:=
\max\left\{
r_{\y},
\sqrt{
\frac{\Phi^0-f_{\min}}{c_\beta N_K}
}
\right\}.
$
In particular there exists
\(k\in\{0,\ldots,K-1\}\) such that $(\x^{k+1},\y^{k+1})$
is an
\begin{enumerate}[label=\normalfont(\roman*)]
\item   \(\mathcal O(K^{-1/3})\)-KKT point if \(r_{\y}=\Theta(K^{-1/3})\),
\(\beta=\Theta(K^{-1/3})\) and \(\xi=\Theta(K^{-2/3})\);
\item 
\(\mathcal O(K^{-1/2})\)-KKT point if \(r_{\y}=0\), \(\beta=\Theta(1)\) and
\(\xi=\Theta(1)\).
\end{enumerate}
\end{theorem}

\begin{remark}[Equality constraints]
The present paper focuses on inequality constraints and does not directly cover general equality constraints. The reason is that equality
multipliers are signed, while our prox-linear augmented Lagrangian model uses the nonnegativity of inequality multipliers to preserve convexity of the primal subproblem. An analogous treatment is possible for equality constraints whose
linearized augmented Lagrangian subproblems remain convex, such as
standard smooth equality constraints. Our analysis can be adapted to this setting, and the resulting argument is
simpler in several respects: There is no complementarity condition, no nonnegative
dual cone, and no active-set switching. The corresponding multiplier
control would rely on a signed constraint qualification condition. We
omit this extension to keep the presentation focused on nonsmooth
nonconvex functional inequalities.
\end{remark}

\section{Proofs of Main Results}

\subsection{Proof of Proposition \ref{prop:dual_eb_local}}

We first establish two auxiliary lemmas. The first shows that the saddle-point
set of the compact dual minimax problem coincides with the KKT pairs of the
original constrained problem whose multipliers lie in the compact dual set;
see Lemma~\ref{lem:S-eq-KKT}. The second shows that, near such a saddle pair,
the augmented active regime is locally stable. A similar result was stated without proof in Proposition~3.1 of \cite{xu2017first}. These lemmas allow us to reduce
the local dual error bound to a sensitivity estimate for a fixed-index lifted
epigraphical nonlinear program.

\begin{proof}[Proof of Lemma \ref{lem:S-eq-KKT}]
Let \((\x^\star,\y^\star)\in\mathcal S\). Since
\(\y^\star\in\mathcal Y(\x^\star)\), the first-order optimality condition
for maximizing the concave function \(d(\cdot,\x^\star)\) over
\(\mathcal Y\) gives
$\bz
\in
-\nabla_{\y} d(\y^\star,\x^\star)
+
\mathcal N_{\mathcal Y}(\y^\star)$.
Moreover, since \(\x^\star=\x(\y^\star,\x^\star)\), Danskin's theorem gives
\begin{equation}
\label{eq:S-grad}
(\nabla_{\y} d(\y^\star,\x^\star))_i
=
\max\left\{
G_i(\x^\star),-\frac{y_i^\star}{\rho}
\right\},
\qquad i\in[d].
\end{equation}
Suppose, to the contrary, that \(\|\y^\star\|_1=R_{\y}\). Then
\[
\mathcal N_{\mathcal Y}(\y^\star)
=
\{\vartheta\mathbf 1-\bm\eta:
\vartheta\ge0,\ \bm\eta\ge\bz,\ \eta_i y_i^\star=0,\ i\in[d]\}.
\]
Thus 
$
\nabla_{\y} d(\y^\star,\x^\star)=\vartheta\mathbf 1-\bm\eta$. 

For every \(i\in\operatorname{supp}(\y^\star)\), we have
\(\eta_i=0\), and therefore
$(\nabla_{\y} d(\y^\star,\x^\star))_i=\vartheta\ge0.
$
Since \(-y_i^\star/\rho<0\), \eqref{eq:S-grad} forces
$
G_i(\x^\star)=\vartheta\ge0$ for
$
i\in\operatorname{supp}(\y^\star).
$
Since \(\x^\star=\x(\y^\star,\x^\star)\), the proximal term has zero
gradient at \(\x=\x^\star\). By Assumption~\ref{ass:struc-regularity}(iii),
\(\mathcal N_{\mathcal X}(\x^\star)=\{\bz\}\). 
We denote \(\mathcal A_i^\star=\mathcal A_i(\x^\star)\) for
\(i=0,1,\ldots,d\).
Hence the primal
stationarity condition gives coefficients
$\mu_{0\ell}\ge0,
\sum_{\ell\in\mathcal A_0^\star}\mu_{0\ell}=1,
$
and, for each \(i\), 
$
\mu_{i\ell}\ge0,
\sum_{\ell\in\mathcal A_i^\star}\mu_{i\ell}=1,
$
such that, with
$
m_i:=(y_i^\star+\rho G_i(\x^\star))_+,
$
we have
\begin{equation}
\label{eq:S-primal-stat}
\bz
=
\sum_{\ell\in\mathcal A_0^\star}
\mu_{0\ell}\nabla\phi_{0\ell}(\x^\star)
+
\sum_{i:\,m_i>0}
\sum_{\ell\in\mathcal A_i^\star}
m_i\mu_{i\ell}\nabla\phi_{i\ell}(\x^\star).
\end{equation}
We claim that every index with \(m_i>0\) belongs to
\(\mathcal I_+(\x^\star)\). Indeed, if \(m_i>0\) and
\(G_i(\x^\star)<0\), then
$
y_i^\star>-\rho G_i(\x^\star)>0.
$ When \(i\in\operatorname{supp}(\y^\star)\),  the previous analysis gives
\(G_i(\x^\star)=\vartheta\ge0\), leading to a contradiction. Hence,  $m_i>0$ implies that $G_i(\x^\star)\ge0$.

Fix \(\ell_0\in\mathcal A_0^\star\). Using
$
\mu_{0\ell_0}
=
1-\sum_{\ell\in\mathcal A_0^\star\setminus\{\ell_0\}}\mu_{0\ell}$,
we rewrite \eqref{eq:S-primal-stat} as
\[
\begin{aligned}
&
\sum_{\ell\in\mathcal A_0^\star\setminus\{\ell_0\}}
\mu_{0\ell}
\bigl(
\nabla\phi_{0\ell}(\x^\star)
-\nabla\phi_{0\ell_0}(\x^\star)
\bigr)
+
\sum_{i:\,m_i>0}
\sum_{\ell\in\mathcal A_i^\star}
m_i\mu_{i\ell}\nabla\phi_{i\ell}(\x^\star)
=
-\nabla\phi_{0\ell_0}(\x^\star).
\end{aligned}
\]
By the standing choice of $L$,
$\|\nabla\phi_{0\ell_0}(\x^\star)\|\le L$.
Since all constraint columns above are included in the collection in
Assumption~\ref{ass:struc-regularity}(ii), the singular-value bound gives
\[
\left(
\sum_{\ell\in\mathcal A_0^\star\setminus\{\ell_0\}}
|\mu_{0\ell}|^2
+
\sum_{i:\,m_i>0}
\sum_{\ell\in\mathcal A_i^\star}
|m_i\mu_{i\ell}|^2
\right)^{1/2}
\le
\frac{\|\nabla\phi_{0\ell_0}(\x^\star)\|}{\tilde\sigma}
\le
\frac{L}{\tilde\sigma}.
\]
 Therefore,
\[
\sum_{i:\,m_i>0}m_i
=
\sum_{i:\,m_i>0}
\sum_{\ell\in\mathcal A_i^\star}m_i\mu_{i\ell}
\le
\frac{\sqrt{N_{\cA}}L}{\tilde\sigma}
<
R_{\y}.
\]
On the other hand, for every \(i\in\operatorname{supp}(\y^\star)\), we have
\(G_i(\x^\star)=\vartheta\ge0\), and hence
$m_i
=
(y_i^\star+\rho\vartheta)_+
=
y_i^\star+\rho\vartheta
\ge
y_i^\star.
$ Thus
$\sum_{i:\,m_i>0}m_i
\ge
\sum_{i\in\operatorname{supp}(\y^\star)}y_i^\star
=
\|\y^\star\|_1
=
R_{\y},
$
which contradicts the previous strict inequality. Hence
$
\|\y^\star\|_1<R_{\y}.
$

Now we identify the KKT condition. Since \(\|\y^\star\|_1<R_{\y}\), the
artificial boundary is inactive, and hence
$
\nabla_{\y}d(\y^\star,\x^\star)
\in
\mathcal N_{\mathbb R^d_+}(\y^\star).
$
Combining this inclusion with \eqref{eq:S-grad}, we obtain the following
componentwise relations. If \(y_i^\star>0\), then the \(i\)-th component of
\(\mathcal N_{\mathbb R^d_+}(\y^\star)\) is \(\{0\}\), so
$
\max\{
G_i(\x^\star),-\frac{y_i^\star}{\rho}
\}=0.
$
Since \(-y_i^\star/\rho<0\), this implies \(G_i(\x^\star)=0\). If
\(y_i^\star=0\), then the \(i\)-th component of
\(\mathcal N_{\mathbb R^d_+}(\y^\star)\) is \((-\infty,0]\), so
$
\max\{G_i(\x^\star),0\}\le 0,
$
and hence \(G_i(\x^\star)\le0\). Therefore,
\[
G(\x^\star)\le\bz,\qquad
y_i^\star G_i(\x^\star)=0,\qquad
(y_i^\star+\rho G_i(\x^\star))_+=y_i^\star,
\qquad i\in[d].
\]
Because \(\x^\star=\x(\y^\star,\x^\star)\), the optimality condition of
\(\x^\star\) for \(F(\cdot,\y^\star,\x^\star)\) gives
\[
\bz
\in
\partial f(\x^\star)
+
\sum_{i=1}^d
(y_i^\star+\rho G_i(\x^\star))_+\partial G_i(\x^\star)
+
\mathcal N_{\mathcal X}(\x^\star).
\]
Using
$
(y_i^\star+\rho G_i(\x^\star))_+=y_i^\star$ for all $i\in[d]$,
and
\(\mathcal N_{\mathcal X}(\x^\star)=\{\bz\}\) by
Assumption~\ref{ass:struc-regularity}(iii), we obtain
$
\bz
\in
\partial f(\x^\star)
+
\sum_{i=1}^d y_i^\star\partial G_i(\x^\star).
$
Together with feasibility, dual feasibility, and complementary slackness, this
is precisely the KKT system of~\eqref{eq:problem}.

Conversely, let \((\x^\star,\y^\star)\) be a KKT pair of~\eqref{eq:problem} with
\(\y^\star\in\mathcal Y\). Feasibility and complementarity give
$
(y_i^\star+\rho G_i(\x^\star))_+=y_i^\star$, $i\in[d]$.
Thus the KKT stationarity condition coincides with the optimality
condition for minimizing \(F(\cdot,\y^\star,\x^\star)\) over
\(\mathcal X\). Since \(F(\cdot,\y^\star,\x^\star)\) is strongly convex,
$
\x^\star=\x(\y^\star,\x^\star)$.
Moreover, by \eqref{eq:S-grad} and complementarity,
$
\nabla_{\y} d(\y^\star,\x^\star)=\bz$.
Hence
$
\bz\in
-\nabla_{\y} d(\y^\star,\x^\star)
+
\mathcal N_{\mathcal Y}(\y^\star)$,
and by concavity of \(d(\cdot,\x^\star)\), this implies
$
\y^\star\in\mathcal Y(\x^\star)$.
Therefore \((\x^\star,\y^\star)\in\mathcal S\). The proof is complete.
\end{proof}

\begin{lemma}[Local stability of the active set]
\label{lemma:active_constant}
Let \((\x^\star,\y^\star)\in\mathcal S\) and suppose that the conditions of
Lemma~\ref{lem:S-eq-KKT} hold. 
Then there exists \(\delta>0\) such that, for any \((\x,\y)\) satisfying
$
\max\{\|\x-\x^\star\|,\|\y-\y^\star\|\}\leq \delta$,
we have
$
\mathcal{I}_\rho(\x,\y)=
\mathcal{I}(\x^\star)$,
where
$
\mathcal{I}_\rho(\x,\y)
:=
\left\{
i\in[d]:
\rho G_i(\x)+y_i>0
\right\}$.
\end{lemma}
\begin{proof}
By Lemma~\ref{lem:S-eq-KKT}, \((\x^\star,\y^\star)\) is a KKT pair of
problem~\eqref{eq:problem}. Hence \(G_i(\x^\star)\le0\), \(y_i^\star\ge0\),
and \(y_i^\star G_i(\x^\star)=0\) for all \(i\in[d]\). Together with
Assumption~\ref{ass:struc-regularity}(i), this implies that, for every
\(i\in\mathcal I(\x^\star)\), we have \(G_i(\x^\star)=0\) and
\(y_i^\star>0\), while for every \(i\notin\mathcal I(\x^\star)\), we have
\(G_i(\x^\star)<0\) and \(y_i^\star=0\).
Thus, the scalar \(\rho G_i(\x^\star)+y_i^\star\) is strictly positive
for \(i\in\mathcal I(\x^\star)\) and strictly negative for
\(i\notin\mathcal I(\x^\star)\).
Choose \(\delta>0\) sufficiently small such that
\[
(\rho L_G+1)\delta
<
\frac12
\min\left\{
\min_{i\in\mathcal I(\x^\star)} y_i^\star,\;
\min_{i\notin\mathcal I(\x^\star)} \bigl(-\rho G_i(\x^\star)\bigr)
\right\},
\]
where an empty minimum is ignored. Then, whenever
\(\max\{\|\x-\x^\star\|,\|\y-\y^\star\|\}\le\delta\), we have
\[
\left|
\bigl(\rho G_i(\x)+y_i\bigr)
-
\bigl(\rho G_i(\x^\star)+y_i^\star\bigr)
\right|
\le
(\rho L_G+1)\delta .
\]
Therefore the sign of \(\rho G_i(\x)+y_i\) is the same as the sign of
\(\rho G_i(\x^\star)+y_i^\star\) for every \(i\in[d]\), 
which proves \(\mathcal I_\rho(\x,\y)=\mathcal I(\x^\star)\).
\end{proof}

\begin{proof}[Proof of Proposition \ref{prop:dual_eb_local}]
 
For simplicity, we denote $\cI:=\cI(\x^\star)$. For any \(i\in[d]\) and  \(\x\), 
$\mathcal A_0(\x)=\{\ell\in\mathcal A_0:\phi_{0\ell}(\x)=f(\x)\}$ and  $\mathcal A_i(\x)=\{\ell\in\mathcal A_i:\phi_{i\ell}(\x)=G_i(\x)\}$. 
By Lemma~\ref{lem:S-eq-KKT}, $(\x^\star,\y^\star)$ is a KKT pair of
\eqref{eq:problem} with
$\|\y^\star\|_1
<R_{\y}$; in particular
the artificial boundary of $\mathcal Y$ is inactive at $\y^\star$.

By strong convexity of the primal subproblem, all maximizers in
$\mathcal Y(\x^\star)$ generate the same primal point $\x^\star$.
The lifted LICQ implied by
Assumption~\ref{ass:struc-regularity}(ii) makes the corresponding
multipliers, and hence the dual multiplier, unique. Thus
$\mathcal Y(\x^\star)=\{\y^\star\}$. Continuity of $d$, compactness of
$\mathcal Y$, and Berge's maximum theorem then imply that
$\y(\z)\to\y^\star$ for any selection
$\y(\z)\in\mathcal Y(\z)$ as $\z\to\x^\star$.
Let $\mathcal U$ be a sufficiently small neighborhood of $\x^\star$.
For each $\z\in\mathcal U$, choose
$\y(\z)\in\mathcal Y(\z)$ and set
$\x^\star(\z):=\x(\y(\z),\z)$. By continuity of $\x(\y,\z)$,
$\x^\star(\z)\to\x^\star$ as $\z\to\x^\star$.

After shrinking $\mathcal U$,
Assumption~\ref{ass:struc-regularity}(i),
Lemma~\ref{lemma:active_constant}, and the local Lipschitz continuity of
$\x(\cdot,\z)$ imply that there exist $\epsilon>0$ and $\gamma>0$,
independent of $\z$, such that for all $\z\in\mathcal U$ and
$\y\in\mathcal Y\cap\mathbb B_\epsilon(\y(\z))$,
we have
$y_i\ge\gamma>0$ for all $i\in\cI$,
$\|\y\|_1<R_{\y}$,
$\x(\y,\z)\in\operatorname{int}(\mathcal X)$, and
$\cI_\rho(\x(\y,\z),\y)=\cI$.
Since each $h_i$ is piecewise affine with finitely many
pieces, the active-piece map is upper semicontinuous. Hence, after
shrinking $\mathcal U$ if  necessary, we have
$
\mathcal A_i(\x)\subseteq\mathcal A_i^\star$ for 
$i=0,1,\ldots,d$ and $\x$ in this neighborhood. In particular, no new affine piece
becomes active near $\x^\star$.

By Danskin's Theorem, we have 
\[
\nabla_{\y}d(\y,\z)_{\cI} = G_{\cI}(\x(\y,\z))\quad
\text{and}\quad \nabla_{\y}d(\y,\z)_{\cI^c}
        =
        -\rho^{-1}\y_{\cI^c}, 
\]        
which implies 
\begin{equation}
\label{eq:dual_bound}
\|G_{\cI}(\x(\y,\z))\| \leq \operatorname{dist}
        \left(
        \bz,
        -\nabla_{\y}d(\y,\z)+\cN_{\R_+^d}(\y)
        \right).
\end{equation}
We next prove the primal error bound
\begin{equation}
\label{eq:eb_x}
        \|\x(\y,\z)-\x(\y(\z),\z)\|
        \le
        c_{\textrm{eb}}\|G_{\cI}(\x(\y,\z))\|. 
\end{equation}
Consider the following lifted epigraphical parametric nonlinear program (NLP): 
\begin{equation}
\label{eq:nlp}
\begin{alignedat}{3}
\min_{\x,t_0,\t_{\cI}}\quad
& \frac{r_{\x}}{2}\|\x-\z\|^2+t_0
\\[0.5ex]
\mathrm{s.t.}\quad
& \phi_{0\ell}(\x)-t_0\le0,
&\qquad& \ell\in\mathcal A_0^\star,
\\
& \phi_{i\ell}(\x)-t_i\le0,
&& i\in\cI,\ \ell\in\mathcal A_i^\star,
\\
& t_i-r_i\le0,
&& i\in\cI .
\end{alignedat}
\end{equation}
The KKT system of \eqref{eq:nlp} is given by
\[
\left\{
\begin{alignedat}{2}
\bz 
&=
r_{\x}(\x-\z)
+
\sum_{\ell\in\mathcal A_0^\star}
\mu_{0\ell}\nabla\phi_{0\ell}(\x)
+
\sum_{i\in\cI}
\sum_{\ell\in\mathcal A_i^\star}
\mu_{i\ell}\nabla\phi_{i\ell}(\x),
&\qquad&
\\
1
&=
\sum_{\ell\in\mathcal A_0^\star}\mu_{0\ell},
\quad
v_i
=
\sum_{\ell\in\mathcal A_i^\star}\mu_{i\ell},
&\quad&
i\in\cI,
\\
0
&\le
\mu_{0\ell}
\perp
t_0-\phi_{0\ell}(\x)
\ge0,
&\qquad&
\ell\in\mathcal A_0^\star,
\\
0
&\le
\mu_{i\ell}
\perp
t_i-\phi_{i\ell}(\x)
\ge0,
&\qquad&
i\in\cI,\ \ell\in\mathcal A_i^\star,
\\
0
&\le
v_i
\perp
r_i-t_i
\ge0,
&\qquad&
i\in\cI .
\end{alignedat}
\right.
\]
Set $\r:=G_{\cI}(\x(\y,\z))$. We first verify that \(\x(\y,\z)\) is the \(\x\)-component of a local KKT point of \eqref{eq:nlp} with this parameter \(\r\).
Let $v_i:=y_i+\rho r_i$ for all $i\in\cI$. By the local active regime stability, after shrinking \(\epsilon\) if necessary, we have $ v_i\ge\gamma/2>0,$ for all $i\in\cI$.  
Since the
active augmented regime is fixed, the optimality condition defining
\(\x(\y,\z)\) gives
\[
\bz \in
r_{\x}(\x(\y,\z)-\z)
+
\partial f(\x(\y,\z))
+
\sum_{i\in \cI}
\bigl(y_i+\rho G_i(\x(\y,\z))\bigr)\partial G_i(\x(\y,\z)).
\]
This gives multipliers
\(\mu_{0\ell}\ge0\) and \(\mu_{i\ell}\ge0\) satisfying the KKT conditions
above, with
$   t_0=f(\x(\y,\z))$ and $
        t_i=r_i=G_i(\x(\y,\z)).$

Now consider the unperturbed problem \eqref{eq:nlp} with \(\r=\bz\). Let \((\bar\x,\bar t_0,\bar\t_{\cI},\bm{\bar\mu},\bm{\bar v})\) be any nearby
local KKT point in the regime $\bar v_i\ge\gamma/2$ for all $i\in \cI$ and $\|(\bar{\bm{v}},\bz)\|_1<R_{\y}$.  Define \(\bm{\nu}\in\R_+^d\) by $\bm{\nu}_{\cI}:=\bar{\bm{v}},$ and $\bm{\nu}_{\cI^c}:=\bz$. Since \(\bar v_i>0\), complementarity gives $\bar t_i=0$ for all $i\in \cI$. Furthermore,
$
        \bar v_i
        =
        \sum_{\ell\in\mathcal A_i^\star}\bar\mu_{i\ell}
        >
        0,
$ 
so for each \(i\in\cI\) at least one epigraphical multiplier
\(\bar\mu_{i\ell}\) is positive. Hence
$
        G_i(\bar\x)=\bar t_i=0,$ for all $i\in\cI$. 
The KKT stationarity of the lifted problem therefore implies 
\[
\bz \in
r_{\x}(\bar\x-\z)
+
\partial f(\bar\x)
+
\sum_{i\in\cI}\bm{\nu}_i\partial G_i(\bar\x).
\]
Equivalently, \(\bar\x\) is the optimal solution of $\min_{\x \in \cX}F(\x,\bm{\nu},\z)$ generated by the dual vector
\(\bm{\nu}\).  Moreover, one verifies directly
$
        \bz
        \in
        -\nabla_{\y}d(\bm{\nu},\z)+\cN_{\R_+^d}(\bm{\nu}).
$ 
Since \(d(\cdot,\z)\) is concave on \(\mathcal Y\), \(\bm{\nu}\) is a dual
maximizer. By the strong convexity of the primal subproblem
\(F(\cdot,\bm{\nu},\z)\), all dual maximizers generate the same primal point. Consequently, the \(\x\)-projection of the nearby unperturbed KKT set of
\eqref{eq:nlp} is the singleton \(\{\x^\star(\z)\}\). 

We now verify the local uniform regularity of the lifted KKT system. First,
Assumption~\ref{ass:struc-regularity}(ii) can imply the reduced LICQ
for the lifted  NLP \eqref{eq:nlp}.  The active gradient rows of the lifted problem at the reference point are \[ \{(\nabla\phi_{0\ell}(\x^\star),-1,\bz):\ell\in\mathcal A_0^\star\} \cup \{(\nabla\phi_{i\ell}(\x^\star),0,-\e_i): i\in\cI,\ \ell\in\mathcal A_i^\star\} \cup \{(\bz,0,\e_i):i\in\cI\}. \]
Eliminating the \(t_0\)-component gives the differences \(\nabla\phi_{0\ell}(\x^\star)-\nabla\phi_{0\ell_0}(\x^\star)\), while eliminating the \(t_i\)-components by the rows \((0,0,\e_i)\) leaves \(\nabla\phi_{i\ell}(\x^\star)\). 

We continue to verify the strong second-order sufficient condition (SSOSC) for the lifted NLP \eqref{eq:nlp}.  Write
$
        \bar\w=(\bar \x,\bar t_0,\bar \t_{\cI})
$ 
and let \(\cL\) be the Lagrangian of \eqref{eq:nlp}.  Since \(\bar t_0\) and \(\bar \t_{\cI}\) enter the lifted problem linearly,
\(\nabla^2_{\w\w}\cL(\bar\w)=\operatorname{diag}(\bm{H}_{\x\x},\bz)\), 
where
\[
        \bm{H}_{\x\x}
        =
        r_{\x}\bm{I}
        +
        \sum_{\ell\in\mathcal A_0(\bar\x)}
        \bar\mu_{0\ell}\nabla^2\phi_{0\ell}(\bar\x)
        +
        \sum_{i\in\cI}
        \sum_{\ell\in\mathcal A_i(\bar\x)}
        \bar\mu_{i\ell}\nabla^2\phi_{i\ell}(\bar\x).
\]
The multipliers are uniformly bounded in the above neighborhood, and the
composite Lagrangian part has weak convexity modulus bounded by
\(L_\rho\). Therefore, by the choice of \(r_{\x}>L_\rho\), we have $ \bm{H}_{\x\x}\succeq (r_{\x}-L_\rho)\bm{I}.$

Let \(\cC(\bar\w,\bm{\bar\mu},\bm{\bar v})\) be the standard critical cone of the lifted NLP at \(\bar\w\). Since every critical direction must be tangent to all active constraints with positive multipliers, we have  \(\cC(\bar\w,\bm{\bar\mu},\bm{\bar v})\subseteq \cC^+(\bar\w,\bm{\bar\mu},\bm{\bar v}), \)  where the critical subspace \(\cC^+(\bar\w,\bm{\bar\mu},\bm{\bar v})\) is defined by
\[ \begin{aligned} 
\cC^+(\bar\w,\bm{\bar\mu},\bm{\bar v}):=\qquad\   \\ 
\Bigl\{ \d=(\d_{\x},\d_{t_0},\d_{\t_{\cI}}): \;& \nabla\phi_{0\ell}(\bar\x)^\top\d_{\x}-\d_{t_0}=\bz \quad \forall \ell\in\mathcal A_0(\bar\x) \text{ with }\bar\mu_{0\ell}>0, \\ & \nabla\phi_{i\ell}(\bar\x)^\top\d_{\x}-\d_{t_i}=\bz \quad \forall i\in\cI,\ \ell\in\mathcal A_i(\bar\x) \text{ with }\bar\mu_{i\ell}>0, \\ & \d_{t_i}=\bz \quad \forall i\in\cI \text{ with }\bar v_i>0 \Bigr\}. 
\end{aligned} \] 
Positive definiteness of the Lagrangian Hessian on the larger subspace
$\cC^+(\bar\w,\bm{\bar\mu},\bm{\bar v})$ can imply the SSOSC required by Robinson's strong
regularity theorem \cite[Theorem 4.1]{Robinson1980}, and is what we now
verify.
For any \(\d\in\cC^+(\bar\w,\bm{\bar\mu},\bm{\bar v})\), since $\sum_{\ell\in\mathcal A_0(\bar\x)}\bar\mu_{0\ell}=1$, we know that $\d_{t_0} = \nabla\phi_{0\hat\ell}(\bar\x)^\top\d_{\x} $ for some $\hat\ell\in\mathcal A_0(\bar\x)$. 
Moreover, since $\bar v_i\ge\gamma/2>0$ for all $i\in\cI$, we have $\d_{t_i}=0$ for all $i\in\cI$. Hence, by local boundedness of \(\nabla\phi_{0\ell}\), there exists \(C_t>0\), independent of \(\z\), such that $\|(\d_{t_0},\d_{\t_{\cI}})\|\le C_t\|\d_{\x}\|$ for any $\d \in \cC^+(\bar\w,\bm{\bar\mu},\bm{\bar v})$. Therefore, for any $\d \in \cC^+(\bar\w,\bm{\bar\mu},\bm{\bar v})$, 
 \[ \langle \d,\nabla^2_{\w\w}\cL(\bar\w)\d\rangle = \d_{\x}^{\top}\bm{H}_{\x\x}\d_{\x} \ge \frac{r_{\x}-L_\rho}{1+C_t^2}\|\d\|^2. \] 
Then, the SSOSC condition  holds at \(\bm{\bar w}\).
Consequently, by 
\cite[Theorem 2.1 \& 4.1]{Robinson1980} (see also
\cite[Theorem 2G.8]{dontchev2009implicit} and
\cite[Section 5.1]{BonnansShapiro2000}), the \(\x\)-projection of the KKT solution map of \eqref{eq:nlp} is locally single-valued and Lipschitz continuous in \(\r\), with a modulus \(c_{\rm eb}>0\)
determined only by \(\tilde\sigma\), \(r_{\x}-L_\rho\), and the local
Lipschitz bounds of \(\nabla\phi_{i\ell}\) (all independent of \(\z\)).
Evaluating at \(\r=G_{\cI}(\x(\y,\z))\), whose unperturbed (\(\r=\bz\))
solution has \(\x\)-component \(\x^\star(\z)\), yields \eqref{eq:eb_x}.
Therefore, armed with \eqref{eq:dual_bound}, we have 
\begin{equation}\label{eq:local-primal-residual-bound-epi}
        \|\x(\y,\z)-\x(\y(\z),\z)\|
        \le
        c_{\rm eb} \operatorname{dist}
        \left(
        \bz,\,
        -\nabla_{\y}d(\y,\z)+\cN_{\R_+^d}(\y)
        \right). 
\end{equation}

Finally, recall that
$     \y_+(\z)
        =
        \operatorname{proj}_{\mathcal Y}
        \bigl(\y+\alpha\nabla_{\y}d(\y,\z)\bigr).
$ 
The projection optimality condition gives
\[
        \alpha^{-1}(\y-\y_+(\z))
        +
        \nabla_{\y}d(\y,\z)
        \in
        \cN_{\R_+^d}(\y_+(\z)), 
\]
which yields 
\[
\begin{aligned}
 &\operatorname{dist}
        \left(
        \bz,
        -\nabla_{\y}d(\y_+(\z),\z)+\cN_{\R_+^d}(\y_+(\z))
        \right)\\
\le\ &
\left\|
\alpha^{-1}(\y-\y_+(\z))
+
\nabla_{\y}d(\y,\z)
-
\nabla_{\y}d(\y_+(\z),\z)
\right\|
\le
\left(\alpha^{-1}+L_d\right)
\|\y-\y_+(\z)\|.
\end{aligned}
\]
Since \(\y_+(\z)\in\mathbb B_\epsilon(\y(\z))\), applying
\eqref{eq:local-primal-residual-bound-epi} at 
\(\y_+(\z)\) yields
\[
\begin{aligned}
\|\x(\y_+(\z),\z)-\x(\y(\z),\z)\|
&\le
c_{\rm eb}
\left(\alpha^{-1}+L_d\right)
\|\y_+(\z)-\y\|.
\end{aligned}
\]
This proves the proposition. 
\end{proof}

\subsection{Proof of Proposition \ref{prop:suff-decrease}}

Now, we focus on the proof of the sufficient-decrease property in Proposition \ref{prop:suff-decrease}. The proof is split into two
cases. When \(r_{\y}>0\), the global dual error bound directly absorbs
the negative term in Proposition~\ref{prop:decrease}. When \(r_{\y}=0\),
we use a two-stage argument: If the negative term dominates, then the
nonhomogeneous dual error bound first implies that the current iterate
enters the local region where Proposition~\ref{prop:dual_eb_local}
applies (see Lemma \ref{lem:local_entry}); otherwise, the positive descent terms already dominate the
negative term.

\begin{proof}[Proof of Lemma \ref{lem:local_entry}]
We first note that the constants in
Proposition~\ref{prop:dual_eb_local} can be chosen uniformly over all
possible limit points in $\mathcal S$. Indeed, $\mathcal S$ is compact,
and the proof of Proposition~\ref{prop:dual_eb_local} shows that
$\mathcal Y(\x^\star)=\{\y^\star\}$ for every
$(\x^\star,\y^\star)\in\mathcal S$. Hence, a finite-cover argument
yields uniform choices of $\epsilon$ and $\omega_2$.

Suppose, to the contrary, that no such \(\bar\delta\) exists. Then there
is a sequence of indices \(k_j\) such that
\[
\max\left\{
\|\x^{k_j}-\x^{k_j+1}\|,
\|\y^{k_j}-\y_+^{k_j}(\z^{k_j})\|,
\|\z^{k_j}-\x^{k_j+1}\|
\right\}
\to0,
\]
but
$
\y^{k_j}\notin\mathbb B_{\epsilon}(\y(\z^{k_j}))$ or 
$\y_+^{k_j}(\z^{k_j})
\notin
\mathbb B_{\epsilon}(\y(\z^{k_j})).
$
By compactness of \(\mathcal X\) and \(\mathcal Y\), after passing to a
subsequence if necessary, we may assume that there exist
$
\bar\x\in\mathcal X$, $\bar\y\in\mathcal Y$
such that
$
\x^{k_j+1}\to\bar\x$,
$\z^{k_j}\to\bar\x$,
$\y^{k_j}\to\bar\y$,
and $
\y_+^{k_j}(\z^{k_j})\to\bar\y .
$ 

By the definition of the projected dual step,
\[
\y_+^{k_j}(\z^{k_j})
=
\proj_{\mathcal Y}
(
\y^{k_j}
+
\alpha\nabla_{\y}d(\y^{k_j},\z^{k_j})
).
\]
Passing to the limit and using the continuity of the projection and of
\(\nabla_{\y}d\), we obtain
$
\bar\y
=
\proj_{\mathcal Y}
(
\bar\y+\alpha\nabla_{\y}d(\bar\y,\bar\x)
)$.
Equivalently,
$
\bz\in
-\nabla_{\y}d(\bar\y,\bar\x)
+
\mathcal{N}_{\mathcal Y}(\bar\y)$.
Since \(d(\cdot,\bar\x)\) is concave on \(\mathcal Y\), this first-order
condition is sufficient for global optimality. Hence
$
\bar\y\in\mathcal Y(\bar\x)$.
Moreover, because
$
\|\x^{k_j+1}-\z^{k_j}\|\to0
$
and
$
\|\x^{k_j}-\x^{k_j+1}\|\to0$,
the primal error bound in Lemma~\ref{prop:lip} gives
$
\|\x^{k_j+1}-\x(\y^{k_j},\z^{k_j})\|\to0$, which implies
$\bar\x=\x(\bar\y,\bar\x)$.
Together with \(\bar\y\in\mathcal Y(\bar\x)\), this means that
\((\bar\x,\bar\y)\) is a KKT pair by Lemma~\ref{lem:S-eq-KKT}.

Therefore, for all sufficiently large \(j\), the pair
\((\y^{k_j},\y_+^{k_j}(\z^{k_j}))\) must lie in the local neighborhood
where Proposition~\ref{prop:dual_eb_local} applies; that is,
$
\y^{k_j},\ \y_+^{k_j}(\z^{k_j})
\in
\mathbb B_{\epsilon}(\y(\z^{k_j})).
$ 
This contradicts the construction. Hence such a \(\bar\delta>0\)
exists. 
\end{proof}

\begin{proof}[Proof of Proposition \ref{prop:suff-decrease}]
We first consider case \textnormal{(i)}. By
Propositions~\ref{prop:decrease} and \ref{prop:dual_eb_KL}, for any \(k\ge0\),
\begin{align*}
\Phi^k-\Phi^{k+1}
\ge\ &
\frac{7}{16\lambda}\|\x^{k}-\x^{k+1}\|^{2}
+
\left(
\frac{1}{8\alpha}
-
28r_{\x}\beta\omega_1^2
\right)
\|\y^{k}-\y_+^{k}(\z^k)\|^2+\frac{4r_{\x}}{7\beta}
\|\z^{k}-\z^{k+1}\|^{2}.
\end{align*}
The condition
$
\beta
\le
\frac{1}{448r_{\x}\omega_1^2\alpha}$
implies
$
\frac{1}{8\alpha}
-
28r_{\x}\beta\omega_1^2
\ge
\frac{1}{16\alpha}$.
Hence, \eqref{eq:suff-decrease} holds.

We next consider case \textnormal{(ii)}. We split the proof into two
cases.
First, suppose that
\begin{align}\label{eq:case1-domination}
&\frac{1}{2}
\max\left\{
\frac{7}{16\lambda}\|\x^k-\x^{k+1}\|^2,\,
\frac{1}{8\alpha}\|\y^k-\y_+^k(\z^k)\|^2,\,
\frac{4r_{\x}}{7\beta}\|\z^k-\z^{k+1}\|^2
\right\}
\notag\\
&<
28r_{\x}\beta
\|\x(\y(\z^k),\z^k)-\x(\y_+^k(\z^k),\z^k)\|^2 .
\end{align}
Then 
using Lemma~\ref{lemma-dual-bd}, we get
\begin{align*}
\|\y^k-\y_+^k(\z^k)\|^2
&\le
448r_{\x}\alpha\beta
\|\x(\y(\z^k),\z^k)-\x(\y_+^k(\z^k),\z^k)\|^2\\
&\le
448r_{\x}\alpha\beta\kappa
\|\y^k-\y_+^k(\z^k)\|,
\end{align*}
which implies
$
\|\y^k-\y_+^k(\z^k)\|
\le
\varrho_1\beta$.
Moreover, since
$
\|\x^{k+1}-\z^k\|
=
\frac{1}{\beta}\|\z^{k+1}-\z^k\|$,
we have
\begin{align*}
\|\x^{k+1}-\z^k\|^2
&\le
98
\|\x(\y(\z^k),\z^k)-\x(\y_+^k(\z^k),\z^k)\|^2\le
98\kappa
\|\y^k-\y_+^k(\z^k)\|\le
\varrho_2^2\beta.
\end{align*}
Similarly,
\begin{align*}
\|\x^k-\x^{k+1}\|^2
&\le
128r_{\x}\lambda\beta
\|\x(\y(\z^k),\z^k)-\x(\y_+^k(\z^k),\z^k)\|^2\\
&\le
128r_{\x}\lambda\kappa\beta
\|\y^k-\y_+^k(\z^k)\|\le
\varrho_3^2\beta^2.
\end{align*}
By the choice of \(\beta\), the above three estimates imply
\[
\max\left\{
\|\x^k-\x^{k+1}\|,
\|\y^k-\y_+^k(\z^k)\|,
\|\x^{k+1}-\z^k\|
\right\}
\le
\bar\delta .
\]
By Lemma~\ref{lem:local_entry}, local dual error bound in
Proposition~\ref{prop:dual_eb_local} applies, and we obtain
\[
\|\x(\y(\z^k),\z^k)-\x(\y_+^k(\z^k),\z^k)\|
\le
\omega_2
\|\y^k-\y_+^k(\z^k)\|.
\]
Consequently,
\begin{align*}
28r_{\x}\beta
\|\x(\y(\z^k),\z^k)-\x(\y_+^k(\z^k),\z^k)\|^2
&\le
28r_{\x}\beta\omega_2^2
\|\y^k-\y_+^k(\z^k)\|^2\le
\frac{1}{16\alpha}
\|\y^k-\y_+^k(\z^k)\|^2,
\end{align*}
where the last inequality follows from
$
\beta
\le
\frac{1}{448r_{\x}\omega_2^2\alpha}$.
Substituting this estimate into Proposition~\ref{prop:decrease}, we
obtain \eqref{eq:suff-decrease}.
On the other hand, if the reverse inequality in \eqref{eq:case1-domination} holds, then \eqref{eq:suff-decrease} follows immediately.
Combining cases (i) and (ii) completes the proof.
\end{proof}

\subsection{Proof of Propositions \ref{prop:primalbd} and \ref{prop:dual-bd-ybd-linearized}}

The purpose of this subsection is to identify sufficiently many iterations at which the primal point is nearly feasible, see Proposition \ref{prop:primalbd},  and the artificial
truncation in the dual set \(\mathcal Y\) is inactive, see Proposition \ref{prop:dual-bd-ybd-linearized}. These estimates
will be used later to convert the algorithmic residuals into the KKT
residuals of problem~\eqref{eq:problem}.

Throughout this subsection, let
$
f_{\min}:=\min_{\x\in\mathcal X} f(\x)
$.
Since \(\bz\in\mathcal Y\), we have
\[
p(\z)
=
\max_{\y\in\mathcal Y}d(\y,\z)
\ge
d(\bz,\z)
=
\min_{\x\in\mathcal X}
\left\{
\mathcal L_{\rho}(\x,\bz)
+
\frac{r_{\x}}{2}\|\x-\z\|^2
\right\}
\ge f_{\min}.
\]
Consequently,
\[
\Phi(\x,\y,\z)
=
F(\x,\y,\z)-d(\y,\z)+p(\z)-d(\y,\z)+p(\z)
\ge
p(\z)
\ge
f_{\min}.
\]

\begin{proof}[Proof of Proposition \ref{prop:primalbd}]
For \(k\ge0\), define
$
\q^{k+1}
:=
\max\{
G(\x^{k+1}),-\frac{\y^k}{\rho}\}.
$
Then the dual update can be
written as
$
\y^{k+1}
=
\proj_{\mathcal Y}
(
\y^k+\alpha(\q^{k+1}-r_{\y}\y^k))$.
We first relate the smoothed augmented Lagrangian term to the projected
dual step. We claim that, for each
\(i\in[d]\),
\begin{equation}\label{eq:scalar-ineq-feasibility}
\frac{\rho}{2}
\left[
\left(G_i(\x^{k+1})+\frac{y_i^{k+1}}{\rho}\right)_+^2
-
\frac{(y_i^{k+1})^2}{\rho^2}
\right]
\ge
y_i^{k+1}q_i^{k+1}
+
\frac{\rho}{2}G_i(\x^{k+1})_+^2 .
\end{equation}
We verify
\eqref{eq:scalar-ineq-feasibility} by considering three cases.

\smallskip
\noindent\emph{Case (i): $G_i(\x^{k+1})\ge0$.}
Then $q_i^{k+1}=G_i(\x^{k+1})$, and equality holds in
\eqref{eq:scalar-ineq-feasibility}.

\smallskip
\noindent\emph{Case (ii): $-y_i^k/\rho\le G_i(\x^{k+1})<0$.}
Then $q_i^{k+1}=G_i(\x^{k+1})$ and $G_i(\x^{k+1})_+=0$. If
$G_i(\x^{k+1})+y_i^{k+1}/\rho\ge0$, the left-hand side of
\eqref{eq:scalar-ineq-feasibility} equals
$y_i^{k+1}G_i(\x^{k+1})+\rho G_i(\x^{k+1})^2/2\ge y_i^{k+1}G_i(\x^{k+1})$.
Otherwise, $G_i(\x^{k+1})<-y_i^{k+1}/\rho$, and the left-hand side equals
$-(y_i^{k+1})^2/(2\rho)\ge y_i^{k+1}G_i(\x^{k+1})$.

\smallskip
\noindent\emph{Case (iii): $G_i(\x^{k+1})<-y_i^k/\rho$.}
Then $q_i^{k+1}=-y_i^k/\rho$, while the projection update gives
$0\le y_i^{k+1}\le y_i^k$. Hence
$G_i(\x^{k+1})<-y_i^{k+1}/\rho$, and the left-hand side of
\eqref{eq:scalar-ineq-feasibility} equals
\[
-\frac{(y_i^{k+1})^2}{2\rho}
\ge
-\frac{y_i^k y_i^{k+1}}{\rho}
=
y_i^{k+1}q_i^{k+1}.
\]
Thus, \eqref{eq:scalar-ineq-feasibility} holds in all three cases.

Summing \eqref{eq:scalar-ineq-feasibility} over \(i=1,\ldots,d\), we get
\[
\mathcal L_{\rho}(\x^{k+1},\y^{k+1})
\ge
f(\x^{k+1})
+
\langle \y^{k+1},\q^{k+1}\rangle
+
\frac{\rho}{2}\|G(\x^{k+1})_+\|^2.
\]
Therefore, using \(f(\x^{k+1})\ge f_{\min}\) and dropping the
nonnegative term
\(\frac{r_{\x}}{2}\|\x^{k+1}-\z^{k+1}\|^2\), we obtain
\begin{align}
F(\x^{k+1},\y^{k+1},\z^{k+1})
&\ge
f_{\min}
+
\frac{\rho}{2}\|G(\x^{k+1})_+\|^2
+
\langle \y^{k+1},\q^{k+1}\rangle
-
\frac{r_{\y}}{2}\|\y^{k+1}\|^2 . \label{eq:F-lower-feasibility}
\end{align}

Since \(\bz\in\mathcal Y\), the convex projection theorem gives
$
\langle
\y^{k+1},
\y^{k+1}-\y^k-\alpha(\q^{k+1}-r_{\y}\y^k)\rangle
\le 0$.
Then
$\langle \y^{k+1},\q^{k+1}\rangle
\ge
\alpha^{-1}
\langle \y^{k+1},\y^{k+1}-\y^k\rangle
+
r_{\y}\langle \y^{k+1},\y^k\rangle$,
and consequently we have
\begin{align}
&\langle \y^{k+1},\q^{k+1}\rangle
-
\frac{r_{\y}}{2}\|\y^{k+1}\|^2 \notag\\
\ge\;&
\alpha^{-1}\langle \y^{k+1},\y^{k+1}-\y^k\rangle
+
r_{\y}
\langle \y^{k+1},\y^k\rangle
-
\frac{r_{\y}}{2}\|\y^{k+1}\|^2 \notag\\
=\;&
\frac{1}{2\alpha}
\left(
\|\y^{k+1}\|^2-\|\y^k\|^2
\right)
+
\frac{1}{2}(\alpha^{-1}-r_{\y})
\|\y^{k+1}-\y^k\|^2
+
\frac{r_{\y}}{2}\|\y^k\|^2, \label{eq:dual-telescope-feasibility}
\end{align}
where the last equality follows from the identities
$
\langle a,a-b\rangle
=
\frac12(
\|a\|^2-\|b\|^2+\|a-b\|^2)$
and
$
\langle a,b\rangle
=
\frac12(
\|a\|^2+\|b\|^2-\|a-b\|^2)$,
with \(a=\y^{k+1}\) and \(b=\y^k\).
Since \(\alpha^{-1}\ge r_{\y}\), the last two terms in
\eqref{eq:dual-telescope-feasibility} are nonnegative except for the
telescoping difference.

Moreover, by Proposition~\ref{prop:suff-decrease}, the sequence
\(\{\Phi^k\}\) is nonincreasing. Since
\[
\Phi(\x,\y,\z)-F(\x,\y,\z)
=
2(p(\z)-d(\y,\z))
\ge0,
\]
for every \(k=0,\ldots,K-1\),  we have
\[
F(\x^{k+1},\y^{k+1},\z^{k+1})
\le
\Phi^{k+1}
\le
\Phi^0.
\]
Combining this with \eqref{eq:F-lower-feasibility} and
\eqref{eq:dual-telescope-feasibility}, and then summing over
\(k=0,\ldots,K-1\), yields
\begin{align*}
&
\frac{\rho}{2}
\sum_{k=0}^{K-1}
\|G(\x^{k+1})_+\|^2+
\sum_{k=0}^{K-1}
\left[
\frac{1}{2\alpha}
\left(
\|\y^{k+1}\|^2-\|\y^k\|^2
\right)
+
\frac{1}{2}(\alpha^{-1}-r_{\y})
\|\y^{k+1}-\y^k\|^2
+
\frac{r_{\y}}{2}\|\y^k\|^2
\right]\\
\le\ &
K(\Phi^0-f_{\min}).
\end{align*}
Using \(\y^0=\bz\) and \(\alpha^{-1}\ge r_{\y}\), the summation term is
nonnegative. Hence
\[
\sum_{k=0}^{K-1}
\|G(\x^{k+1})_+\|^2
\le
\frac{2K(\Phi^0-f_{\min})}{\rho},
\]
and the number of indices satisfying
\(\|G(\x^{k+1})_+\|>\delta\) is at most
$
\frac{2K(\Phi^0-f_{\min})}{\rho\delta^2}$.
The proof is complete.
\end{proof}

\begin{proof}[Proof of Proposition \ref{prop:dual-bd-ybd-linearized}]
By Definition~\ref{ass:compact}, whenever
\(\|\y\|_1<R_{\y}\), the artificial boundary of \(\mathcal Y\) is inactive
and
$
\mathcal{N}_{\mathcal Y}(\y)=\mathcal{N}_{\mathbb R^d_+}(\y)$.
For \(\bar\x\in\mathcal X\), define the prox-linear constraint model $G_{\bar\x,i}(\x) := h_i(c_i(\bar\x)+\nabla c_i(\bar\x)^\top(\x-\bar\x))$, $\forall i\in[d].$ 
The uniform bound \eqref{eq:gd_upper} is part of
Assumption~\ref{ass:basic}.

By the proof of Proposition~\ref{prop:suff-decrease} (each case there yields the following termwise bound before taking the minimum $c_\beta$), we have
\[
\Phi^k-\Phi^{k+1}
\ge
\frac{3}{16\lambda}\|\x^{k}-\x^{k+1}\|^{2}
+
\frac{1}{16\alpha}\|\y^k-\y_+^k(\z^k)\|^2
+
\frac{2r_{\x}}{7\beta}\|\z^k-\z^{k+1}\|^2 .
\]
Since
$
\z^{k+1}-\z^k=\beta(\x^{k+1}-\z^k)$,
and \(\beta\le \tfrac{1}{28}\), we obtain
\[
\Phi^k-\Phi^{k+1}
\ge
\frac{2\beta}{7}
\left[
\lambda^{-1}\|\x^k-\x^{k+1}\|^2
+
\alpha^{-1}\|\y^k-\y_+^k(\z^k)\|^2
+
r_{\x}\|\z^k-\x^{k+1}\|^2
\right].
\]
Summing over \(k=0,\ldots,K-1\), and using \(\Phi^K\ge f_{\min}\), gives
\[
\sum_{k=0}^{K-1}
\left[
\lambda^{-1}\|\x^{k+1}-\x^k\|^2
+
\alpha^{-1}\|\y^k-\y_+^k(\z^k)\|^2
+
r_{\x}\|\x^{k+1}-\z^k\|^2
\right]
\le
\frac{7(\Phi^0-f_{\min})}{2\beta}.
\]
Therefore, the number of indices satisfying
\[
\max\left\{
\lambda^{-1}\|\x^{k+1}-\x^k\|^2,\,
\alpha^{-1}\|\y^k-\y_+^k(\z^k)\|^2,\,
r_{\x}\|\x^{k+1}-\z^k\|^2
\right\}
>
\xi
\]
is at most
$
\frac{7(\Phi^0-f_{\min})}{2\beta\xi}$.
Thus, for at least
$
K-\frac{7(\Phi^0-f_{\min})}{2\beta\xi}
$
indices, we have
\begin{equation}\label{eq:small-move-good}
\lambda^{-1}\|\x^{k+1}-\x^k\|^2\le\xi,
\quad
\alpha^{-1}\|\y^k-\y_+^k(\z^k)\|^2\le\xi,
\quad \text{and} \quad 
r_{\x}\|\x^{k+1}-\z^k\|^2\le\xi.
\end{equation}

By Proposition~\ref{prop:primalbd}, the condition
$
\|G(\x^{k+1})_+\|\le \delta
$
holds for at least
$
K-
\frac{2K(\Phi^0-f_{\min})}{\rho\delta^2}
$
indices. Combining the two counting estimates, there are at least
$
K
-
\frac{7(\Phi^0-f_{\min})}{2\beta\xi}
-
\frac{2K(\Phi^0-f_{\min})}{\rho\delta^2}
$
indices satisfying both \eqref{eq:small-move-good} and
$
\x^{k+1}\in
\mathcal R_{\delta}
\subseteq
\mathcal R_{\delta_{\rm cq}}$.
The condition \eqref{eq:rho-K-good} guarantees that the above number is
positive.

Fix such an index \(k\), and define
$
\gamma_i^k
:=
\bigl(y_i^k+\rho G_{\x^k,i}(\x^{k+1})\bigr)_+$,
$\mathcal I^k
:=
\mathcal I_{\rm cq}(\x^k,\x^{k+1})$.
By \eqref{eq:gd_upper}, Lemma~\ref{prop:lip}, and
\eqref{eq:small-move-good},
\[
|G_{\x^k,i}(\x^{k+1})
-
G_i(\x(\y^k,\z^k))|\le 2L\|\x^{k+1}-\x^k\|+L\|\x^{k+1}-\x(\y^k,\z^k)\|
\le
L(2+\zeta)\sqrt{\lambda\xi}.
\]
Thus, for every \(i\notin\mathcal I^k\),
\[
G_i(\x(\y^k,\z^k))
<
-\delta_{\rm cq}
+
L(2+\zeta)\sqrt{\lambda\xi}
\le
-\sqrt{\xi/\alpha}.
\]
Let
$
q_i
:=
\max\{
G_i(\x(\y^k,\z^k)),-y_i^k/\rho\}
-r_{\y}y_i^k$.
If \(y_i^k>\rho\sqrt{\xi/\alpha}\), then
\(q_i<-\sqrt{\xi/\alpha}\). Since
$
y_{+,i}^k
\le
(y_i^k+\alpha q_i)_+$,
we obtain, using \(\alpha\le\rho\),
\[
\|\y^k-\y_+^k(\z^k)\|
\ge
y_i^k-y_{+,i}^k
\ge
\min\{\alpha|q_i|,y_i^k\}
>
\sqrt{\alpha\xi},
\]
which contradicts \eqref{eq:small-move-good}. Hence
$
y_i^k\le\rho\sqrt{\xi/\alpha}$, $i\notin\mathcal I^k$.
Consequently,
\[
y_i^k+\rho G_{\x^k,i}(\x^{k+1})
<
\rho\sqrt{\xi/\alpha}-\rho\delta_{\rm cq}
\le0,
\]
and therefore \(\gamma_i^k=0\) for every
\(i\notin\mathcal I^k\). Thus
$
\operatorname{supp}(\bm\gamma^k)
\subseteq
\mathcal I^k$.
Since
$
\x^{k+1}\in\mathcal R_{\delta}
\subseteq\mathcal R_{\delta_{\rm cq}}$,
$\|\x^{k+1}-\x^k\|
\le
\sqrt{\lambda\xi}
\le
\delta_{\rm cq}$,
Assumption~\ref{ass:ucq} applies with
\(\bar\x=\x^k\), \(\x=\x^{k+1}\), and
\(\bm\mu=\bm\gamma^k\).

Using $\g_i^k\in\partial G_{\x^k,i}(\x^{k+1})$ for $i\in [d]$, we obtain 
$
\bar\sigma\|\bm{\gamma}^k\|
\le
\dist(
\bz,
\sum_{i=1}^d\gamma_i^k\g_i^k
+
\mathcal N_{\mathcal X}(\x^{k+1}))$.
Moreover, the  optimality condition of the primal update \eqref{eq:prox_linear_model} gives
\[
\bz \in
\v_0^k
+
\sum_{i=1}^d\gamma_i^k\g_i^k
+
\lambda^{-1}(\x^{k+1}-\x^k)
+
r_{\x}(\x^{k+1}-\z^k)
+
\mathcal N_{\mathcal X}(\x^{k+1}),
\]
where 
$
f_0^k(\x)
:=
h_0(c_0(\x^k)+\nabla c_0(\x^k)^\top(\x-\x^k))$ and
$
\v_0^k\in\partial f_0^k(\x^{k+1})$. 
Hence
\[
\dist\left(
\bz,
\sum_{i=1}^d\gamma_i^k\g_i^k
+
\mathcal N_{\mathcal X}(\x^{k+1})
\right)
\le
\|\v_0^k\|
+
\lambda^{-1}\|\x^{k+1}-\x^k\|
+
r_{\x}\|\x^{k+1}-\z^k\|.
\]
By the boundedness of the subgradients of the linearized composite
objective from \eqref{eq:gd_upper}, \(\|\v_0^k\|\le L\). Hence, using \eqref{eq:small-move-good},
 we obtain
\[
\|\bm{\gamma}^k\|
\le
\frac{
L+(\sqrt{\lambda^{-1}}+\sqrt{r_{\x}})\sqrt{\xi}
}{
\bar\sigma
}.
\]

We next bound \(\y^k\). For \(i\in\mathcal I^k\), since
\(G_{\x^k,i}(\x^{k+1})\ge-\delta_{\rm cq}\) and \(a\le a_+\), we have
\[
y_i^k
\le
\gamma_i^k-\rho G_{\x^k,i}(\x^{k+1})
\le
\gamma_i^k+\rho\delta_{\rm cq}.
\]
For \(i\notin\mathcal I^k\), the preceding argument gives
\(\gamma_i^k=0\) and
\(y_i^k\le\rho\sqrt{\xi/\alpha}\le\rho\delta_{\rm cq}\).
Thus,
\[
\|\y^k\|_1
\le\sum_{i\in\mathcal I^k}\gamma_i^k+\rho d\delta_{\rm cq}
\le\sqrt d\,\|\bm\gamma^k\|+\rho d\delta_{\rm cq}
\le \rho d\delta_{\rm cq}
+\frac{\sqrt d\,(L+(\sqrt{\lambda^{-1}}+\sqrt{r_{\x}})\sqrt\xi)}{\bar\sigma}.
\]

It remains to verify that the projection in the dual update does not
hit the artificial boundary of \(\mathcal Y\). Define
$
\u^k
:=
\y^k+\alpha
(
\max\{G(\x^{k+1}),-\frac{\y^k}{\rho}\}
-r_{\y}\y^k)$ and $\y^{k+1}=\proj_{\mathcal Y}(\u^k)$.
For all \(i\in[d]\), since \(y_i^k\ge0\) and \(r_{\y}\ge0\), we claim that 
$
(u_i^k)_+
\le
y_i^k+\alpha G_i(\x^{k+1})_+ .
$
If
$
G_i(\x^{k+1})\ge -\frac{y_i^k}{\rho}$,
then
\[
u_i^k
=
y_i^k+\alpha\left(G_i(\x^{k+1})-r_{\y}y_i^k\right)
\le
y_i^k+\alpha G_i(\x^{k+1})
\le
y_i^k+\alpha G_i(\x^{k+1})_+ .
\]
If
$
G_i(\x^{k+1})<-\frac{y_i^k}{\rho}$,
then
$
u_i^k
=
y_i^k+\alpha\left(-\frac{y_i^k}{\rho}-r_{\y}y_i^k\right)
\le
y_i^k
\le
y_i^k+\alpha G_i(\x^{k+1})_+$.
For the counted indices, \(\|G(\x^{k+1})_+\|\le\delta\). Therefore,
\[
\begin{aligned}
\|(\u^k)_+\|_1
&\le
\|\y^k\|_1+\alpha\|G(\x^{k+1})_+\|_1\\ 
&\le
\rho d\delta_{\rm cq}
+
\frac{
\sqrt{d}(L+(\sqrt{\lambda^{-1}}+\sqrt{r_{\x}})\sqrt{\xi})
}{
\bar\sigma
}
+
\alpha\sqrt d\delta
<
R_{\y},
\end{aligned}
\]
where we used \(\|G(\x^{k+1})_+\|_1\le\sqrt d\delta\) and
\eqref{eq:ryrequire}.
Together with the preceding estimate for \(\|\y^k\|_1\), we obtain
$
\|\y^k\|_1<R_{\y}$ and $
\|\y^{k+1}\|_1<R_{\y}
$
for all counted indices. The proof is complete.
\end{proof}

\subsection{Proof of Theorem \ref{thm:general}}

The following lemma converts the algorithmic residuals selected from
the sufficient decrease estimate into the KKT residual of the original
problem.

\begin{lemma}
\label{lemma-episolcol}
Let
\(\epsilon\ge0\). Suppose that
$
\|\y^k\|_1<R_{\y}$, $ 
\|\y^{k+1}\|_1<R_{\y},
$ 
and
\[
\max\left\{
\|\x^{k+1}-\x^k\|,
\|\y_+^k(\z^k)-\y^k\|,
\|\x^{k+1}-\z^k\|,
r_{\y}
\right\}
\le
\epsilon .
\]
Then, there exists a constant \(\tau>0\) such that
$
(\x^{k+1},\y^{k+1})$
is a \(\tau\epsilon\)-KKT point of problem~\eqref{eq:problem} in the
sense of Definition~\ref{defi:primal-dual}.
\end{lemma}

\begin{proof}
We first relate the actual dual iterate \(\y^{k+1}\) to the ideal dual
step \(\y_+^k(\z^k)\). Due to the definitions of \(\y^{k+1}\) and
\(\y_+^k(\z^k)\), and  the nonexpansiveness of the projection,
\begin{align*}
\|\y^{k+1}-\y_+^k(\z^k)\|
&\le
\alpha
\left\|
\nabla_{\y}F(\x^{k+1},\y^k,\z^k)
-
\nabla_{\y}d(\y^k,\z^k)
\right\| \\
&=
\alpha
\left\|
\nabla_{\y}\mathcal L_\rho(\x^{k+1},\y^k)
-
\nabla_{\y}\mathcal L_\rho(\x(\y^k,\z^k),\y^k)
\right\| \\
&\le
\alpha L_G
\|\x^{k+1}-\x(\y^k,\z^k)\| \\
&\le
\alpha L_G\zeta
\|\x^{k+1}-\x^k\|\le
\alpha L_G\zeta\,\epsilon,
\end{align*}
where the second inequality uses the Lipschitz continuity, and
the third inequality follows from Lemma~\ref{prop:lip}. Hence
\begin{equation}\label{eq:yk1-yk-bound}
\|\y^{k+1}-\y^k\|
\le
\|\y^{k+1}-\y_+^k(\z^k)\|
+
\|\y_+^k(\z^k)-\y^k\|
\le
(1+\alpha L_G\zeta)\epsilon.
\end{equation}

Now, we bound the primal stationarity residual. By
Lemma~\ref{lemma-dualdiff}, we have
\[
\nabla_{\z}d(\y^{k+1},\x^{k+1})
=
r_{\x}
(
\x^{k+1}
-
\x(\y^{k+1},\x^{k+1})).
\]
Hence,
\begin{align*}
\|\nabla_{\z}d(\y^{k+1},\x^{k+1})\| 
&=
r_{\x}
\|\x^{k+1}-\x(\y^{k+1},\x^{k+1})\|\\
&\le
r_{\x}
(
\|\x^{k+1}-\x(\y^k,\z^k)\|
+
\|\x(\y^k,\z^k)-\x(\y^{k+1},\z^k)\|\\
&\quad  +
\|\x(\y^{k+1},\z^k)
-
\x(\y^{k+1},\x^{k+1})\|
).
\end{align*}
Therefore, using Lemma~\ref{prop:lip}, Lemma~\ref{lemma-sollip},
and \eqref{eq:yk1-yk-bound}, we get
\[
\|\nabla_{\z}d(\y^{k+1},\x^{k+1})\|
\le
r_{\x}
\left[
\zeta
+
\sigma_2(1+\alpha L_G\zeta)
+
\sigma_1
\right]\epsilon .
\]

Next, we bound the dual residual. By the update rule 
$
\y^{k+1}
=
\proj_{\mathcal Y}
(
\y^k+\alpha\nabla_{\y}F(\x^{k+1},\y^k,\z^k))$, the projection optimality condition gives
\[
\frac{1}{\alpha}
(
\y^k-\y^{k+1}
)
+
\nabla_{\y}F(\x^{k+1},\y^k,\z^k)
\in
\mathcal{N}_{\mathcal Y}(\y^{k+1}).
\]
Due to 
$
\mathcal{N}_{\mathcal Y}(\y^{k+1})
=
\mathcal{N}_{\mathbb R^d_+}(\y^{k+1})$,
we have
\begin{align*}
&\dist\left(
\bz,
-\nabla_{\y}\mathcal L_\rho(\x^{k+1},\y^{k+1})
+
\mathcal{N}_{\mathbb R^d_+}(\y^{k+1})
\right)\\
\le\;&
\alpha^{-1}\|\y^{k+1}-\y^k\|
+
\left\|
\nabla_{\y}F(\x^{k+1},\y^k,\z^k)
-
\nabla_{\y}\mathcal L_\rho(\x^{k+1},\y^{k+1})
\right\|.
\end{align*}
Recall that
\[
\nabla_{\y}F(\x^{k+1},\y^k,\z^k)
=
\nabla_{\y}\mathcal L_\rho(\x^{k+1},\y^k)
-r_{\y}\y^k\
\text{and}\
\nabla_{\y}\mathcal L_\rho(\x,\y)
=
\max\left\{
G(\x),-\frac{\y}{\rho}
\right\}.
\]
Hence,
\[
\left\|
\nabla_{\y}\mathcal L_\rho(\x^{k+1},\y^k)
-
\nabla_{\y}\mathcal L_\rho(\x^{k+1},\y^{k+1})
\right\|
\le
\rho^{-1}\|\y^{k+1}-\y^k\|.
\]
Using \(\|\y^k\|\le R_{\y}\), we obtain
\begin{align*}
\dist\left(
\bz,
-\nabla_{\y}\mathcal L_\rho(\x^{k+1},\y^{k+1})
+
\mathcal{N}_{\mathbb R^d_+}(\y^{k+1})
\right)
&\le
(\alpha^{-1}+\rho^{-1})\|\y^{k+1}-\y^k\|
+
R_{\y}r_{\y}\\
&\le
\left[
(\alpha^{-1}+\rho^{-1})(1+\alpha L_G\zeta)
+
R_{\y}
\right]\epsilon,
\end{align*}
where the last inequality follows from \eqref{eq:yk1-yk-bound} and
\(r_{\y}\le\epsilon\).
Taking
\[
\tau
:=
\max\left\{
r_{\x}
\left[
\zeta
+
\sigma_2(1+\alpha L_G\zeta)
+
\sigma_1
\right],
\,
(\alpha^{-1}+\rho^{-1})(1+\alpha L_G\zeta)+R_{\y}
\right\}
\]
completes the proof.
\end{proof}

\begin{proof}[Proof of Theorem \ref{thm:general}]
Let \(\mathcal{J}_K\subseteq\{0,\ldots,K-1\}\) be the set of good
indices provided by Proposition~\ref{prop:dual-bd-ybd-linearized}. Then
$
|\mathcal{J}_K|\ge N_K>0,
$ 
and for every \(k\in\mathcal{J}_K\), we have 
$\|\y^k\|_1<R_{\y}$ and 
$
\|\y^{k+1}\|_1<R_{\y}.
$

By Proposition~\ref{prop:suff-decrease} and the update
$
\z^{k+1}-\z^k=\beta(\x^{k+1}-\z^k)$,
we have
\[
\Phi^k-\Phi^{k+1}
\ge
c_\beta
\left(
\|\x^{k}-\x^{k+1}\|^2
+
\|\y^{k}-\y_+^k(\z^k)\|^2
+
\|\x^{k+1}-\z^k\|^2
\right).
\]
Summing over \(k\in\mathcal{J}_K\), and using
\(\Phi^K\ge f_{\min}\), gives
\[
\Phi^0-f_{\min}
\ge
c_\beta
\sum_{k\in\mathcal{J}_K}
\left(
\|\x^{k}-\x^{k+1}\|^2
+
\|\y^{k}-\y_+^k(\z^k)\|^2
+
\|\x^{k+1}-\z^k\|^2
\right).
\]
Hence there exists \(k\in\mathcal{J}_K\) such that
\begin{equation}\label{eq:selected-good-index}
\max\left\{
\|\x^{k}-\x^{k+1}\|,
\|\y^{k}-\y_+^k(\z^k)\|,
\|\x^{k+1}-\z^k\|
\right\}
\le
\sqrt{
\frac{\Phi^0-f_{\min}}{c_\beta N_K}
}
\le
\varepsilon_K .
\end{equation}
By the definition of \(\varepsilon_K\), we also have
$
r_{\y}\le \varepsilon_K$.
Combining this with \eqref{eq:selected-good-index}, we obtain
\[
\max\left\{
\|\x^{k}-\x^{k+1}\|,
\|\y^{k}-\y_+^k(\z^k)\|,
\|\x^{k+1}-\z^k\|,
r_{\y}
\right\}
\le
\varepsilon_K .
\]
Therefore, Lemma~\ref{lemma-episolcol} implies that
$
(\x^{k+1},\y^{k+1})$
is an \(\mathcal O(\varepsilon_K)\)-KKT point of
problem~\eqref{eq:problem}.

For the first rate statement, since
$
r_{\y}=\Theta(K^{-1/3})$,
$\beta=\Theta(K^{-1/3})$,
and $
\xi=\Theta(K^{-2/3})$,
we know that
$
c_\beta=\Theta(\beta)=\Theta(K^{-1/3})$,
and hence
$
\sqrt{
\frac{\Phi^0-f_{\min}}{c_\beta N_K}
}
=
\mathcal O(K^{-1/3})$.
Together with \(r_{\y}=\Theta(K^{-1/3})\), this gives
$
\varepsilon_K=\mathcal O(K^{-1/3})$.
For the second rate statement, under $r_{\y}=0$, $\beta=\Theta(1)$, and $\xi=\Theta(1)$, we have $c_\beta=\Theta(1)$. Hence,
$
\varepsilon_K
=
\mathcal O(K^{-1/2})$.
The proof is complete.
\end{proof}

\section*{Acknowledgments}
Jiajin Li was supported by a Natural Sciences and Engineering Research Council of Canada Discovery Grant RGPIN-2025-05817.

\bibliographystyle{plain}
\bibliography{ref}

@article{el2025complexity,
  title={Complexity of Linearized Perturbed Augmented {L}agrangian Methods for Nonsmooth Nonconvex Optimization with Nonlinear Equality Constraints},
  author={Bourkhissi, Lahcen El and Necoara, Ion and Patrinos, Panagiotis and Tran-Dinh, Quoc},
  journal={arXiv preprint arXiv:2503.01056},
  year={2025}
}

@article{kong2023iteration,
  title={Iteration complexity of a proximal augmented {L}agrangian method for solving nonconvex composite optimization problems with nonlinear convex constraints},
  author={Kong, Weiwei and Melo, Jefferson G. and Monteiro, Renato DC},
  journal={Mathematics of Operations Research},
  volume={48},
  number={2},
  pages={1066--1094},
  year={2023},
  publisher={INFORMS}
}

@article{Robinson1980,
  author  = {Robinson, Stephen M.},
  title   = {Strongly Regular Generalized Equations},
  journal = {Mathematics of Operations Research},
  volume  = {5},
  number  = {1},
  pages   = {43--62},
  year    = {1980}
}

@book{BonnansShapiro2000,
  author    = {Bonnans, J. Fr{\'e}d{\'e}ric and Shapiro, Alexander},
  title     = {Perturbation Analysis of Optimization Problems},
  publisher = {Springer},
  year      = {2000}
}

@article{andreani2012relaxed,
  title={A relaxed constant positive linear dependence constraint qualification and applications},
  author={Andreani, Roberto and Haeser, Gabriel and Schuverdt, Maria Laura and Silva, Paulo JS},
  journal={Mathematical Programming},
  volume={135},
  number={1},
  pages={255--273},
  year={2012},
  publisher={Springer}
}

@article{andreani2025relaxed,
  title={A relaxed quasinormality condition and the boundedness of dual augmented {L}agrangian sequences},
  author={Andreani, Roberto and Haeser, Gabriel and Schuverdt, Maria Laura and Secchin, Leonardo Delarmelina},
  journal={SIAM Journal on Optimization},
  volume={35},
  number={4},
  pages={2474--2489},
  year={2025},
  publisher={SIAM}
}

@article{lewis2016proximal,
  title={A proximal method for composite minimization},
  author={Lewis, Adrian S. and Wright, Stephen J.},
  journal={Mathematical Programming},
  volume={158},
  number={1},
  pages={501--546},
  year={2016},
  publisher={Springer}
}

@article{li2026smoothing,
  title={Smoothing Meets Perturbation: Unified and Tight Analysis for Nonconvex-Concave Minimax Optimization},
  author={Li, Jiajin and Nagarajan, Mahesh and Pan, Siyu and Zhang, Nanxi},
  journal={arXiv preprint arXiv:2602.14185},
  year={2026}
}

@article{rockafellar2000optimization,
  title={Optimization of conditional value-at-risk},
  author={Rockafellar, R. Tyrrell and Uryasev, Stanislav},
  journal={Journal of Risk},
  volume={2},
  pages={21--42},
  year={2000}
}

@article{liu2025single,
  title   = {A {SPIDER}-type stochastic subgradient method for expectation-constrained nonconvex nonsmooth optimization},
author  = {Liu, Wei and Xu, Yangyang},
  journal = {SIAM Journal on Optimization},
  volume  = {36},
  number  = {2},
  pages   = {1125--1153},
  year    = {2026},
  publisher={SIAM}
}

@article{yang2025single,
  author  = {Yang, Ming and Li, Gang and Hu, Quanqi and Lin, Qihang and Yang, Tianbao},
  title   = {Single-loop algorithms for stochastic nonconvex optimization with weakly convex constraints},
  journal = {Transactions on Machine Learning Research},
  year    = {2026}
}

@article{he2023newton,
  title={A {N}ewton-{CG} based augmented {L}agrangian method for finding a second-order stationary point of nonconvex equality constrained optimization with complexity guarantees},
  author={He, Chuan and Lu, Zhaosong and Pong, Ting Kei},
  journal={SIAM Journal on Optimization},
  volume={33},
  number={3},
  pages={1734--1766},
  year={2023},
  publisher={SIAM}
}

@article{jia2025first,
  title={First-order methods for nonsmooth nonconvex functional constrained optimization with or without {S}later points},
  author={Jia, Zhichao and Grimmer, Benjamin},
  journal={SIAM Journal on Optimization},
  volume={35},
  number={2},
  pages={1300--1329},
  year={2025},
  publisher={SIAM}
}

@article{dahal2026damped,
  title={Damped proximal augmented {L}agrangian method for weakly-convex problems with convex constraints},
  author={Dahal, Hari and Liu, Wei and Xu, Yangyang},
  journal={Mathematical Programming Computation},
  pages={1--50},
  year={2026},
  publisher={Springer}
}

@article{huang2023oracle,
  title={Oracle complexity of single-loop switching subgradient methods for non-smooth weakly convex functional constrained optimization},
  author={Huang, Yankun and Lin, Qihang},
  journal={Advances in Neural Information Processing Systems},
  volume={36},
  pages={61327--61340},
  year={2023}
}

@article{drezner2019planar,
  title={The planar multiple obnoxious facilities location problem: A {V}oronoi based heuristic},
  author={Drezner, Zvi and Kalczynski, Pawel and Salhi, Said},
  journal={Omega},
  volume={87},
  pages={105--116},
  year={2019},
  publisher={Elsevier}
}

@article{li2024stochastic,
  title={Stochastic inexact augmented {L}agrangian method for nonconvex expectation constrained optimization},
  author={Li, Zichong and Chen, Pin-Yu and Liu, Sijia and Lu, Songtao and Xu, Yangyang},
  journal={Computational Optimization and Applications},
  volume={87},
  number={1},
  pages={117--147},
  year={2024},
  publisher={Springer}
}

@article{lin2022complexity,
  title={Complexity of an inexact proximal-point penalty method for constrained smooth non-convex optimization},
  author={Lin, Qihang and Ma, Runchao and Xu, Yangyang},
  journal={Computational Optimization and Applications},
  volume={82},
  number={1},
  pages={175--224},
  year={2022},
  publisher={Springer}
}

@inproceedings{ma2020quadratically,
  title={Quadratically regularized subgradient methods for weakly convex optimization with weakly convex constraints},
  author={Ma, Runchao and Lin, Qihang and Yang, Tianbao},
  booktitle={International Conference on Machine Learning},
  pages={6554--6564},
  year={2020},
  organization={PMLR}
}

@article{boob2023stochastic,
  title={Stochastic first-order methods for convex and nonconvex functional constrained optimization},
  author={Boob, Digvijay and Deng, Qi and Lan, Guanghui},
  journal={Mathematical Programming},
  volume={197},
  number={1},
  pages={215--279},
  year={2023},
  publisher={Springer}
}

@book{rockafellar2009variational,
  title={Variational Analysis},
  author={Rockafellar, R Tyrrell and Wets, Roger J-B},
  volume={317},
  year={2009},
  publisher={Springer Science \& Business Media}
}

@article{huang2025inexact,
  title={Inexact {M}oreau Envelope {L}agrangian Method for Non-Convex Constrained Optimization under Local Error Bound Conditions on Constraint Functions},
  author={Huang, Yankun and Lin, Qihang and Xu, Yangyang},
  journal={arXiv preprint arXiv:2502.19764},
  year={2025}
}

@article{zhang2020single,
	title={A single-loop smoothed gradient descent-ascent algorithm for nonconvex-concave min-max problems},
	author={Zhang, Jiawei and Xiao, Peijun and Sun, Ruoyu and Luo, Zhiquan},
	journal={Advances in Neural Information Processing Systems},
	volume={33},
	pages={7377--7389},
	year={2020}
}

@book{ben2009robust,
  title={Robust Optimization},
  author={Ben-Tal, Aharon and El Ghaoui, Laurent and Nemirovski, Arkadi},
  volume={28},
  year={2009},
  publisher={Princeton University Press}
}

@article{bertsimas2011theory,
  title={Theory and applications of robust optimization},
  author={Bertsimas, Dimitris and Brown, David B. and Caramanis, Constantine},
  journal={SIAM Review},
  volume={53},
  number={3},
  pages={464--501},
  year={2011},
  publisher={SIAM}
}

@article{davis2019stochastic,
  title={Stochastic model-based minimization of weakly convex functions},
  author={Davis, Damek and Drusvyatskiy, Dmitriy},
  journal={SIAM Journal on Optimization},
  volume={29},
  number={1},
  pages={207--239},
  year={2019},
  publisher={SIAM}
}

@article{drusvyatskiy2019efficiency,
  title={Efficiency of minimizing compositions of convex functions and smooth maps},
  author={Drusvyatskiy, Dmitriy and Paquette, Courtney},
  journal={Mathematical Programming},
  volume={178},
  pages={503--558},
  year={2019},
  publisher={Springer}
}

@inproceedings{alacaoglu2024complexity,
  title={Complexity of single loop algorithms for nonlinear programming with stochastic objective and constraints},
  author={Alacaoglu, Ahmet and Wright, Stephen J},
  booktitle={International Conference on Artificial Intelligence and Statistics},
  pages={4627--4635},
  year={2024},
  organization={PMLR}
}

@article{bolte2018nonconvex,
  title={Nonconvex {L}agrangian-based optimization: monitoring schemes and global convergence},
  author={Bolte, J{\'e}r{\^o}me and Sabach, Shoham and Teboulle, Marc},
  journal={Mathematics of Operations Research},
  volume={43},
  number={4},
  pages={1210--1232},
  year={2018},
  publisher={INFORMS}
}

@article{pu2024smoothed,
  title={Smoothed Proximal {L}agrangian Method for Nonlinear Constrained Programs},
  author={Pu, Wenqiang and Sun, Kaizhao and Zhang, Jiawei},
  journal={arXiv preprint arXiv:2408.15047},
  year={2024}
}

@article{xie2021complexity,
  title={Complexity of proximal augmented {L}agrangian for nonconvex optimization with nonlinear equality constraints},
  author={Xie, Yue and Wright, Stephen J.},
  journal={Journal of Scientific Computing},
  volume={86},
  pages={1--30},
  year={2021},
  publisher={Springer}
}

@article{zhu2024first,
  title={A first-order primal-dual method for nonconvex constrained optimization based on the augmented {L}agrangian},
  author={Zhu, Daoli and Zhao, Lei and Zhang, Shuzhong},
  journal={Mathematics of Operations Research},
  volume={49},
  number={1},
  pages={125--150},
  year={2024},
  publisher={INFORMS}
}

@article{rockafellar1973dual,
  author  = {Rockafellar, R. Tyrrell},
  title   = {A dual approach to solving nonlinear programming problems by unconstrained optimization},
  journal = {Mathematical Programming},
  volume  = {5},
  number  = {1},
  pages   = {354--373},
  year    = {1973}
}

@article{boyd2011distributed,
  author  = {Boyd, Stephen and Parikh, Neal and Chu, Eric and Peleato, Borja and Eckstein, Jonathan},
  title   = {Distributed Optimization and Statistical Learning via the Alternating Direction Method of Multipliers},
  journal = {Foundations and Trends in Machine Learning},
  volume  = {3}, number = {1}, pages = {1--122}, year = {2011}
}

@article{chambolle2011first,
  author  = {Chambolle, Antonin and Pock, Thomas},
  title   = {A First-Order Primal-Dual Algorithm for Convex Problems with Applications to Imaging},
  journal = {Journal of Mathematical Imaging and Vision},
  volume  = {40}, number = {1}, pages = {120--145}, year = {2011}
}

@article{hestenes1969multiplier,
  title={Multiplier and gradient methods},
  author={Hestenes, Magnus R.},
  journal={Journal of Optimization Theory and Applications},
  volume={4},
  number={5},
  pages={303--320},
  year={1969},
  publisher={Springer}
}

@article{powell1969method,
  title={A method for nonlinear constraints in minimization problems},
     author    = {Powell, Michael J. D.},
  journal={Optimization},
  pages={283--298},
  year={1969},
  publisher={Academic Press}
}

@article{li2025nonsmooth,
  title={Nonsmooth nonconvex--nonconcave minimax optimization: Primal--dual balancing and iteration complexity analysis},
  author={Li, Jiajin and Zhu, Linglingzhi and So, Anthony Man-Cho},
  journal={Mathematical Programming},
  volume={214},
  number={1-2},
  pages={591--641},
  year={2025},
  publisher={Springer Berlin Heidelberg Berlin/Heidelberg}
}

@article{zhu2026primal,
  title={Primal-Dual Methods for Nonsmooth Nonconvex Optimization with Orthogonality Constraints},
  author={Zhu, Linglingzhi and Ding, Wentao and Liu, Shangyuan and So, Anthony Man-Cho},
  journal={arXiv preprint arXiv:2604.04130},
  year={2026}
}

@article{xu2017first,
  title={First-order methods for constrained convex programming based on linearized augmented {L}agrangian function},
  author={Xu, Yangyang},
  journal={INFORMS Journal on Optimization},
  volume={3},
  number={1},
  pages={89--117},
  year={2021},
  publisher={INFORMS}
}

@book{dontchev2009implicit,
  title={Implicit Functions and Solution Mappings},
  author={Dontchev, Asen L. and Rockafellar, R. Tyrrell},
  year={2009},
  publisher={Springer}
}

@article{burke2020strong,
  title={Strong metric (sub)regularity of {K}arush--{K}uhn--{T}ucker mappings for piecewise linear-quadratic convex-composite optimization and the quadratic convergence of {N}ewton's method},
  author={Burke, James V. and Engle, Abraham},
  journal={Mathematics of Operations Research},
  volume={45},
  number={3},
  pages={1164--1192},
  year={2020},
  publisher={INFORMS}
}

\appendix

\raggedbottom
\section{Technical Lemmas and Proofs} \label{sec:lemmas}

By assumption of the problem \eqref{eq:problem} (recalling $L\ge L_hL_c$ and $r_{\x}>L_\rho$) with Lemma \ref{lem: weakcvx} and \cite[Lemma 3.2, Lemma 4.2]{drusvyatskiy2019efficiency}, we directly have the following useful result.

\begin{fact}\label{fact-2}
Let $\y\in\mathcal{Y}$. For all $\x, \bar{\x} \in\mathcal{X}$ it follows that
\[
-\frac{L_{\rho}+\lambda^{-1}}{2}\|\x-\bar{\x}\|^{2} \leq \mathcal{L}_{\rho}(\x,\y)-F_{\bar{\x},\lambda}(\x, \y) \leq \frac{L_{\rho}-\lambda^{-1}}{2}\|\x-\bar{\x}\|^{2}.
\]
\end{fact}

The following three lemmas are needed in our analysis. Their proofs are
analogous to those of \cite[Lemmas~A.2--A.3 and Proposition~2]{li2025nonsmooth},
respectively.

\begin{lemma}\label{lemma-sollip}
For any
\(\y, \y^{\prime} \in \mathcal{Y}\) and
\(\z, \z^{\prime} \in \R^{n}\), the following inequalities hold:
    \begin{align}
        &\|\x(\y, \z)-\x(\y, \z^{\prime})\|
        \leq \sigma_{1}\|\z-\z^{\prime}\|, \label{lip-z}\\
        &\|\x^\star(\z)-\x^\star(\z^{\prime})\|
        \leq \sigma_{1}\|\z-\z^{\prime}\|, \label{lip-starz}\\
        &\|\x(\y, \z)-\x(\y^{\prime}, \z)\|
        \leq \sigma_{2}\|\y-\y^{\prime}\|,
        \label{lip-y}
    \end{align}
where $\sigma_{1}, \sigma_{2}$ are defined in \eqref{eq:constants}.
\end{lemma}

\begin{lemma}\label{lemma-dualdiff} 
    The dual function \(d(\cdot,\cdot)\) is differentiable on
    \(\mathcal{Y}\times \R^{n}\), and for each
    \(\y\in\mathcal{Y}\), \(\z\in\R^{n}\),
    \begin{align*}
    \nabla_{\y} d(\y,\z)
    &=
    \nabla_{\y} F(\x(\y,\z),\y,\z) =
    \max\left\{
    G(\x(\y,\z)),
    -\frac{\y}{\rho}
    \right\}
    -r_{\y}\y,\\
    \nabla_{\z} d(\y,\z)
    &=
    \nabla_{\z} F(\x(\y, \z),\y, \z)
    =
    r_{\x}(\z-\x(\y,\z)).
    \end{align*}
 Moreover, $\nabla d(\cdot,\cdot)$ is Lipschitz continuous, i.e.,
    \begin{align}
    &\|\nabla_{\y} d(\y^{\prime},\z)
    -\nabla_{\y} d(\y^{\prime \prime},\z)\|
    \leq
    L_{d}\|\y^{\prime}-\y^{\prime \prime}\|,
    \quad
    \text{for all } \y^{\prime}, \y^{\prime \prime} \in \mathcal{Y},
    \notag\\
    &\|\nabla_{\z} d(\y,\z')
    -\nabla_{\z} d(\y,\z'')\|
    \leq
    L_{d}'\|\z^{\prime}-\z^{\prime \prime}\|,
    \quad
    \text{for all } \z^{\prime}, \z^{\prime \prime} \in \R^{n},
    \notag
    \end{align}
    with
    $
    L_{d}:=L_G\sigma_{2}+\frac{1}{\rho}+r_{\y}$ and
    $L_{d}':=(\sigma_1+1)r_{\x}$.
\end{lemma}

\begin{lemma}[Lipschitz-type primal error bound]
\label{prop:lip}
For any \(k\ge0\),
it holds that
\[
\|\x^{k+1}-\x(\y^{k}, \z^{k})\|
\leq
\zeta\|\x^{k}-\x^{k+1}\|,
\]
where $\zeta>0$ is the constant defined in Proposition  \ref{prop:decrease}.
\end{lemma}

\subsection{Proof of Lemma \ref{lem:local-cq-classical}}
\label{proof:local-cq-classical}
For (i), necessity follows by setting
$\bar\x=\x=\x^\star$ in Assumption~\ref{ass:ucq}. To prove sufficiency, suppose that \eqref{eq:MFCQ_imply}
holds at every feasible point but
Assumption~\ref{ass:ucq} fails.
For each $j\ge1$, apply this failure with
$\delta_{\rm cq}=\bar\sigma=j^{-1}$.
Since the normal cone is a cone, we may divide both
the violating multiplier and its normal vector by
the multiplier norm.
Thus, there exist $\x^j,\bar\x^j\in\mathcal X$,
$\bm\mu^j\ge\bz$ with $\|\bm\mu^j\|=1$, and
$\v^j\in\mathcal N_{\mathcal X}(\x^j)$ such that
\[
\begin{gathered}
\|G(\x^j)_+\|\le j^{-1},
\quad
\|\x^j-\bar\x^j\|\le j^{-1},
\quad \left\|
\sum_{i=1}^d\mu_i^j\nabla G_i(\bar\x^j)+\v^j
\right\|<j^{-1},
\end{gathered}
\]
where $\mu_i^j=0$ whenever
\[
G_i(\bar\x^j)
+\langle\nabla G_i(\bar\x^j),\x^j-\bar\x^j\rangle
<-j^{-1}.
\]The gradients are bounded on $\mathcal X$, so
$\|\bm\mu^j\|=1$ and the residual bound imply that
$\{\v^j\}$ is bounded.
By compactness and $\|\x^j-\bar\x^j\|\to0$,
we may pass to a subsequence such that
$\x^j,\bar\x^j\to\x^\star\in\mathcal X$, $ \bm\mu^j\to\bm\mu^\star\ge\bz$ and $\v^j\to\v^\star$. 
Continuity gives $G(\x^\star)\le\bz$ and
$\|\bm\mu^\star\|=1$.
If $G_i(\x^\star)<0$, then
\[
G_i(\bar\x^j)
+\langle\nabla G_i(\bar\x^j),\x^j-\bar\x^j\rangle
\to G_i(\x^\star)<0.
\]
The support restriction therefore gives $\mu_i^j=0$
for all sufficiently large $j$, and hence
$\mu_i^\star=0$.
Thus, $\mu_i^\star G_i(\x^\star)=0$ for every $i\in[d]$.
Moreover, the closed graph of $\mathcal N_{\mathcal X}$ gives
$\v^\star\in\mathcal N_{\mathcal X}(\x^\star)$.
Taking the limit in the residual bound yields
\[
\bz
=
\sum_{i=1}^d\mu_i^\star\nabla G_i(\x^\star)+\v^\star
\in
\sum_{i=1}^d\mu_i^\star\nabla G_i(\x^\star)
+\mathcal N_{\mathcal X}(\x^\star).
\]
Condition~\eqref{eq:MFCQ_imply} then implies
$\bm\mu^\star=\bz$, contradicting
$\|\bm\mu^\star\|=1$.

For (ii), let $\x'\in\mathcal X$ satisfy
$G_i(\x')<0$ for every $i\in[d]$.
By (i), it suffices to verify \eqref{eq:MFCQ_imply}
at every feasible point.
Fix a feasible $\x^\star$ and let $\bm\mu$ satisfy
the premises of \eqref{eq:MFCQ_imply}.
Choose $\v\in\mathcal N_{\mathcal X}(\x^\star)$
such that
$\sum_{i=1}^d\mu_i\nabla G_i(\x^\star)+\v=\bz$.
By the definition of $\mathcal N_{\mathcal X}(\x^\star)$,
we have $\langle\v,\x'-\x^\star\rangle\le0$
due to $\x'\in\mathcal X$.
Using 
the convexity of $G_i$, and
$\mu_iG_i(\x^\star)=0$, we obtain
\[
\begin{aligned}
0
&\le
\sum_{i=1}^d\mu_i
\langle\nabla G_i(\x^\star),\x'-\x^\star\rangle
\le
\sum_{i=1}^d\mu_i
\bigl(G_i(\x')-G_i(\x^\star)\bigr)
=
\sum_{i=1}^d\mu_iG_i(\x').
\end{aligned}
\]
Since $\bm\mu\ge\bz$ and $G_i(\x')<0$ for every $i$,
this forces $\bm\mu=\bz$, proving
\eqref{eq:MFCQ_imply}.

For (iii), the function $\psi$ is continuously differentiable,
with
$
\nabla\psi(\x)
=
\sum_{i=1}^d G_i(\x)_+\nabla G_i(\x).
$ 
Fix $\x\in\mathcal R_{\delta_{\rm cq}}$ and take
$\bar\x=\x$ and $\bm\mu=G(\x)_+$.
Since $\mu_i>0$ only if $G_i(\x)>0$,
the support restriction in Assumption~\ref{ass:ucq}
is satisfied.
Applying the assumption gives
\[
\bar\sigma\|G(\x)_+\|
\le
\dist\!\left(
\bz,
\sum_{i=1}^d G_i(\x)_+\nabla G_i(\x)
+\mathcal N_{\mathcal X}(\x)
\right)
=
\dist\!\left(
\bz,\nabla\psi(\x)+\mathcal N_{\mathcal X}(\x)
\right).
\]
Squaring and dividing by two yields
\eqref{eq:local-feasibility-PL}, since
$\|G(\x)_+\|^2=2\psi(\x)$.
We complete the proof.

\subsection{Proof of Lemma \ref{lem: weakcvx}}
\label{proof:lemma_lrho}
Under Assumption \ref{ass:basic}, the standard convex composite
weak convexity rule \cite[Lemma 4.2]{drusvyatskiy2019efficiency} implies that \(f=h_0\circ c_0\) is
\(L\)-weakly convex on \(\mathcal X\). Therefore, it remains to
show that
$
\x\mapsto
\|
(G(\x)+\frac{\y}{\rho})_+\|^2$
is weakly convex with the nonsmooth composite structure.

First, for each \(i\in[d]\), we know that
$
u\mapsto (h_i(u)+\frac{y_i}{\rho})_+$
is nonnegative and convex since \(h_i\) is convex. 
Then the map
$
u\mapsto (h_i(u)+\frac{y_i}{\rho})_+^2$
is convex. Indeed, \(u\mapsto h_i(u)+\frac{y_i}{\rho}\) is convex, the
map \(s\mapsto s_+\) is convex and nondecreasing, and \(s\mapsto s^2\) is
convex and nondecreasing on \(\mathbb R_+\). Hence the claim follows from
the standard composition rule for convex functions.
Define
$
M_i(\y):=
\max_{\x,\bar\x\in\mathcal X}
(h_i(c_i(\bar\x)+\nabla c_i(\bar\x)^\top(\x-\bar\x))+\frac{y_i}{\rho})_+$.
Since \(\mathcal X\) is compact, \(M_i(\y)<+\infty\). Moreover,
 the convex function
$
u\mapsto
\frac{1}{2}
(h_i(u)+\frac{y_i}{\rho})_+^2
$
is \(L_hM_i(\y)\)-Lipschitz. Applying again the
convex-composite weak convexity rule, we obtain that
$
\x\mapsto
\frac{1}{2}
(h_i(c_i(\x))+\frac{y_i}{\rho})_+^2$
is \(LM_i(\y)\)-weakly convex on \(\mathcal X\).
Since \(f\) is \(L\)-weakly convex, the function
\(\mathcal L_\rho(\cdot,\y)\) is
$
L(1+\rho\sum_{i=1}^d M_i(\y))
$-weakly convex on \(\mathcal X\). It remains to bound this modulus uniformly. Since
\(\y\in\mathcal Y\subseteq\R^d_+\), we have
$
M_i(\y)
\le
\max_{\x,\bar\x\in\mathcal X}|h_i(c_i(\bar\x)+\nabla c_i(\bar\x)^\top(\x-\bar\x))|+\frac{y_i}{\rho}$.
Therefore,
\[
\rho\sum_{i=1}^d M_i(\y)
\le
\rho\sum_{i=1}^d\max_{\x,\bar\x\in\mathcal X}|h_i(c_i(\bar\x)+\nabla c_i(\bar\x)^\top(\x-\bar\x))|
+\sum_{i=1}^d y_i
\le
\rho R_{\x}+R_{\y}.
\]
Thus \(\mathcal L_\rho(\cdot,\y)\) is
$
L(1+R_{\y}+\rho R_{\x})
$-weakly convex. 
The nonsmooth composite structure follows from the
representation
$
\mathcal L_\rho(\x,\y)
=
H_{\y}(c_0(\x),c_1(\x),\ldots,c_d(\x))
-\frac{1}{2\rho}\|\y\|^2$,
where
$
H_{\y}(u_0,u_1,\ldots,u_d)
:=
h_0(u_0)
+
\frac{\rho}{2}
\sum_{i=1}^d
(h_i(u_i)+\frac{y_i}{\rho})_+^2$.
The function \(H_{\y}\) is convex, since each of its component
terms is convex.

For the last claim, fix \(\x\). For each \(i\), the function
$
y_i\mapsto
\frac{\rho}{2}
[
(G_i(\x)+\frac{y_i}{\rho})_+^2
-
\frac{y_i^2}{\rho^2}
]$
is concave and continuously differentiable. 
Hence \(\mathcal L_\rho(\x,\cdot)\) is concave and continuously
differentiable on \(\mathcal Y\). The proof is complete.

\section{Proof Details of Basic Descent Property}
\label{sec:suff_decrease}

\begin{lemma}[Primal Descent]
\label{lemma-F-des} 
For any \(k\ge0\), it follows that
\begin{align*}
& F(\x^{k}, \y^{k},\z^k)-F(\x^{k+1}, \y^{k+1},\z^{k+1})\\
\ge\ &
\frac{2\lambda^{-1}+r_{\x}-L_\rho}{2}\|\x^k-\x^{k+1}\|^2 +
\left\langle
\max\left\{G(\x^{k+1}), -\frac{\y^k}{\rho}\right\}
-r_{\y}\y^k,
\y^{k}-\y^{k+1}
\right\rangle\\
&+\frac{r_{\y}}{2}\|\y^k-\y^{k+1}\|^2+ \frac{(2-\beta)r_{\x}}{2 \beta}\|\z^{k}-\z^{k+1}\|^{2}. 
\end{align*}
\end{lemma}

\begin{lemma}[Dual Ascent]
\label{lemma-d-aes} 
For any \(k\ge0\), it follows that
\[
\begin{aligned}
&d(\y^{k+1},\z^{k+1})-d(\y^{k},\z^k)\\
\ge\ &
\left\langle
\max\left\{
G(\x(\y^{k},\z^k)),
-\frac{\y^k}{\rho}
\right\}
-r_{\y}\y^k,
\y^{k+1}-\y^{k}
\right\rangle -\frac{L_d}{2}\|\y^k-\y^{k+1}\|^2 \\
&
+\frac{r_{\x}}{2}
\left\langle
\z^{k+1}-\z^{k},
\z^{k+1}+\z^{k}
-2 \x(\y^{k+1}, \z^{k+1})
\right\rangle .
\end{aligned}
\]
\end{lemma}

\begin{lemma}[Proximal Descent]
\label{lemma-p-des} 
For any \(k\ge0\), it follows that
\[
p(\z^k)-p(\z^{k+1})
\ge
\frac{r_{\x}}{2}
\left\langle
\z^{k+1}-\z^{k},
2 \x(\y(\z^{k+1}), \z^{k})-\z^{k}-\z^{k+1}
\right\rangle. 
\]
\end{lemma}

\paragraph{Proof of Proposition~\ref{prop:decrease}}
From Lemmas~\ref{lemma-F-des}, \ref{lemma-d-aes}, and
\ref{lemma-p-des} in Appendix~\ref{sec:suff_decrease}, we know that\begin{align*}
&\Phi(\x^k,\y^k,\z^k)
-
\Phi(\x^{k+1},\y^{k+1},\z^{k+1})\\
={}&
F(\x^k,\y^k,\z^k)
-
F(\x^{k+1},\y^{k+1},\z^{k+1})
+
2\bigl(
d(\y^{k+1},\z^{k+1})-d(\y^k,\z^k)
\bigr)\\
&+
2\bigl(p(\z^k)-p(\z^{k+1})\bigr)
\displaybreak[3]\\
\ge{}&
\frac{2\lambda^{-1}+r_{\x}-L_\rho}{2}
\|\x^k-\x^{k+1}\|^2
+
\frac{(2-\beta)r_{\x}}{2\beta}
\|\z^k-\z^{k+1}\|^2\\
&-
\left(L_d-\frac{r_{\y}}{2}\right)
\|\y^k-\y^{k+1}\|^2\\
&+
\underbrace{
\left\langle
\nabla_{\y}F(\x^{k+1},\y^k,\z^k)
-
2\nabla_{\y}F(\x(\y^k,\z^k),\y^k,\z^k),
\y^k-\y^{k+1}
\right\rangle
}_{\text{\ding{172}}}\\
&+
\underbrace{
2r_{\x}
\left\langle
\z^{k+1}-\z^k,
\x(\y(\z^{k+1}),\z^k)
-
\x(\y^{k+1},\z^{k+1})
\right\rangle
}_{\text{\ding{173}}}.
\end{align*}
Subsequently, we simplify the terms \text{\ding{172}} and
\text{\ding{173}}. First, for \text{\ding{172}}, we know that
\begin{align*}
\text{\ding{172}}
=\ &
\left\langle
\nabla_{\y}F(\x^{k+1},\y^k,\z^k),
\y^{k+1}-\y^k
\right\rangle \\
&+
2\left\langle
\nabla_{\y}F(\x(\y^k,\z^k),\y^k,\z^k)
-
\nabla_{\y}F(\x^{k+1},\y^k,\z^k),
\y^{k+1}-\y^k
\right\rangle .
\end{align*}
For the first term, one has that
$\left\langle
\nabla_{\y}F(\x^{k+1}, \y^{k},\z^k),
\y^{k+1}-\y^{k}
\right\rangle 
\ge
\frac{1}{\alpha}\|\y^{k}-\y^{k+1}\|^2$,
which follows from the update of dual variables. On the other hand,
for the second term we have
\begin{align*}
&2\left\langle
\nabla_{\y} F(\x(\y^{k}, \z^{k}), \y^{k}, \z^{k})
-
\nabla_{\y} F(\x^{k+1}, \y^{k},\z^k),
\y^{k+1}-\y^{k}
\right\rangle\\
\ge\ &
-2L_G\|\x^{k+1}-\x(\y^{k}, \z^{k})\|
\|\y^{k}-\y^{k+1}\|\\
\ge\ &
-L_G\zeta^2\|\y^{k}-\y^{k+1}\|^{2}
-L_G\|\x^{k+1}-\x^{k}\|^{2},
\end{align*}
where the last inequality follows from Lemma~\ref{prop:lip} and
\(2|a||b|\leq \tau a^2+\tau^{-1}b^2\). Together, we obtain
\[
\text{\ding{172}}
\ge
\left(\frac{1}{\alpha}-L_G\zeta^2\right)
\|\y^{k}-\y^{k+1}\|^{2}
-
L_G\|\x^{k+1}-\x^{k}\|^{2}. 
\]
Then, we continue to bound \(\text{\ding{173}}\):
\[
\begin{aligned}
\text{\ding{173}}
=\ &
2 r_{\x}
\left\langle
\z^{k+1}-\z^{k},
\x(\y(\z^{k+1}), \z^{k})
-\x(\y^{k+1}, \z^{k+1})
\right\rangle \\
=\ &
2 r_{\x}
\left\langle
\z^{k+1}-\z^{k},
\x(\y(\z^{k+1}), \z^{k})
-\x(\y(\z^{k+1}), \z^{k+1})
\right\rangle \\
&+
2 r_{\x}
\left\langle
\z^{k+1}-\z^{k},
\x(\y(\z^{k+1}), \z^{k+1})
-\x(\y^{k+1}, \z^{k+1})
\right\rangle \\
\ge\ &
-2r_{\x}\sigma_1\|\z^{k+1}-\z^k\|^2 +
2 r_{\x}
\left\langle
\z^{k+1}-\z^{k},
\x(\y(\z^{k+1}), \z^{k+1})
-\x(\y^{k+1}, \z^{k+1})
\right\rangle \\
\ge\ &
-2r_{\x}\sigma_1\|\z^{k+1}-\z^k\|^2
-\frac{r_{\x}}{7\beta}\|\z^{k+1}-\z^k\|^2 \\
&-
7r_{\x}\beta
\|\x(\y(\z^{k+1}), \z^{k+1})
-\x(\y^{k+1}, \z^{k+1})\|^2,
\end{aligned}
\]
where the inequality follows from \eqref{lip-z}, the Cauchy--Schwarz
inequality, and the AM--GM inequality. Thus,
\begin{equation}
\label{desas-keymid}
\begin{aligned}
& \Phi(\x^k,\y^k,\z^k)
-
\Phi(\x^{k+1},\y^{k+1},\z^{k+1}) \\
\geq\ &
\frac{2\lambda^{-1}+r_{\x}-L_\rho-2L_G}{2}
\|\x^{k}-\x^{k+1}\|^{2}
+
\left(
\frac{1}{\alpha}
-L_G\zeta^2
-L_d
+\frac{r_{\y}}{2}
\right)
\|\y^{k}-\y^{k+1}\|^{2} \\
&+
\left(
\frac{(2-\beta)r_{\x}}{2 \beta}
-2r_{\x}\sigma_1
-\frac{r_{\x}}{7\beta}
\right)
\|\z^{k}-\z^{k+1}\|^{2} \\
&-
7r_{\x}\beta
\|\x(\y(\z^{k+1}),\z^{k+1})
-\x(\y^{k+1}, \z^{k+1})\|^2. 
\end{aligned}
\end{equation}
On top of
$\|\y^{k+1}-\y_{+}^{k}(\z^k)\| 
\leq
\alpha L_G\|\x^{k+1}-\x(\y^k, \z^k)\| 
\leq
\alpha L_G \zeta\|\x^k-\x^{k+1}\|$,
we have with
$
\eta:=\alpha L_G\zeta$,
\begin{equation}
\label{desas-final-key1}
\begin{aligned}
\|\y^{k+1}-\y^{k}\|^{2} 
&=
\|\y^{k+1}-\y_{+}^{k}(\z^{k})
+\y_{+}^{k}(\z^{k})-\y^{k}\|^{2} \\
&\ge
\frac{1}{2}\|\y^{k}-\y_{+}^{k}(\z^{k})\|^{2}
-\|\y^{k+1}-\y_{+}^{k}(\z^{k})\|^{2} \\
&\ge
\frac{1}{2}\|\y^{k}-\y_{+}^{k}(\z^{k})\|^{2}
-\eta^2\|\x^{k}-\x^{k+1}\|^2 .
\end{aligned}
\end{equation}
On the other hand, by Lemma~\ref{lemma-sollip} 
we have
\begin{equation}
\label{desas-final-key2}
\begin{aligned}
&\|\x(\y(\z^{k+1}),\z^{k+1})
-\x(\y^{k+1}, \z^{k+1})\|^{2}\\
\leq\ &
4\|\x(\y(\z^{k+1}),\z^{k+1})
-\x(\y(\z^{k}),\z^{k})\|^{2} +
4\|\x(\y(\z^{k}),\z^{k})
-\x(\y_+^{k}(\z^k), \z^{k})\|^{2} \\
&+
4\|\x(\y_+^{k}(\z^k), \z^{k})
-\x(\y^{k+1}, \z^{k})\|^{2} +
4\|\x(\y^{k+1}, \z^{k})
-\x(\y^{k+1}, \z^{k+1})\|^{2} \\
\leq\ &
8\sigma_1^2\|\z^{k}-\z^{k+1}\|^2
+
4\|\x(\y(\z^{k}),\z^{k})
-\x(\y_+^{k}(\z^k), \z^{k})\|^{2} +
4\sigma_2^2\eta^2\|\x^k-\x^{k+1}\|^2.
\end{aligned}
\end{equation}
Substituting \eqref{desas-final-key1} and \eqref{desas-final-key2}
into \eqref{desas-keymid} yields
\begin{align*}
&\Phi(\x^k,\y^k,\z^k)
-
\Phi(\x^{k+1},\y^{k+1},\z^{k+1})\\
\geq\ &
\left(
\frac{2\lambda^{-1}+r_{\x}-L_\rho-2L_G}{2}
-28r_{\x}\beta\sigma_2^2\eta^2
\right)
\|\x^{k}-\x^{k+1}\|^{2}\\
&+
\left(
\frac{1}{\alpha}
-L_G\zeta^2
-L_d
+\frac{r_{\y}}{2}
\right)
\left(
\frac{1}{2}
\|\y^{k}-\y^{k}_+(\z^{k})\|^{2}
-\eta^2\|\x^{k+1}-\x^k\|^2
\right)\\
&+
\left(
\frac{(2-\beta)r_{\x}}{2 \beta}
-2r_{\x}\sigma_1
-\frac{r_{\x}}{7\beta}
-56r_{\x}\beta\sigma_1^2
\right)
\|\z^{k}-\z^{k+1}\|^{2}\\
&-
28r_{\x}\beta
\|\x(\y(\z^{k}),\z^{k})
-\x(\y_+^{k}(\z^k), \z^{k})\|^{2}.
\end{align*}
By the choices of $r_{\x}$, $\lambda$, $\alpha$, $\beta$,
we know \(\sigma_1\le 3/2\), and the following estimates hold:
\begin{itemize}
\item Since
$
\frac{1}{\alpha}
-L_G\zeta^2
-L_d
+\frac{r_{\y}}{2}
\ge
\frac{1}{2\alpha}$,
we have
$
\frac{1}{2}
\left(
\frac{1}{\alpha}
-L_G\zeta^2
-L_d
+\frac{r_{\y}}{2}
\right)
\ge
\frac{1}{4\alpha}
\ge
\frac{1}{8\alpha}$.

\item Since \(\beta\le 1/28\) and \(\sigma_1\le 3/2\), we have
\[
\begin{aligned}
&\frac{(2-\beta)r_{\x}}{2 \beta}
-2r_{\x}\sigma_1
-\frac{r_{\x}}{7\beta}
-56r_{\x}\beta\sigma_1^2\ge
\frac{6r_{\x}}{7\beta}
-\frac{7r_{\x}}{2}
-126r_{\x}\beta
\ge
\frac{4r_{\x}}{7\beta}.
\end{aligned}
\]

\item Since \(\lambda^{-1}\ge L_G\) and
\(\eta=\alpha L_G\zeta\), the condition
\(\alpha\le 1/(8L_G\zeta^2)\) implies
$
\frac{\eta^2}{\alpha}
=
\alpha L_G^2\zeta^2
\le
\frac{1}{8\lambda}$.
Moreover, since
$
\beta\le \frac{1}{14\alpha r_{\x}\sigma_2^2}$,
we have
$
28r_{\x}\beta\sigma_2^2\eta^2
\le
\frac{2\eta^2}{\alpha}
\le
\frac{1}{4\lambda}$.
Using also
$
\frac{1}{\alpha}
-L_G\zeta^2
-L_d
+\frac{r_{\y}}{2}
\le
\frac{1}{\alpha}$,
we obtain
\[
\begin{aligned}
&\frac{2\lambda^{-1}+r_{\x}-L_\rho-2L_G}{2}
-28r_{\x}\beta\sigma_2^2\eta^2
-
\left(
\frac{1}{\alpha}
-L_G\zeta^2
-L_d
+\frac{r_{\y}}{2}
\right)\eta^2\\
\ge\ &
\frac{1}{\lambda}
-\frac{1}{4\lambda}
-\frac{1}{8\lambda}
\ge
\frac{7}{16\lambda}.
\end{aligned}
\]
\end{itemize}
Combining these estimates completes the proof.

\end{document}